\documentclass{amsart}
\usepackage[utf8]{inputenc}
 \usepackage{graphics, fullpage,color, epsfig,url}
\usepackage{tikz,amsthm,amsmath,amstext,amssymb,amscd,esint,mathrsfs, pspicture,multicol,graphpap,graphics,graphicx,comment,enumerate,subfig,sidecap,wrapfig,color,pict2e}
\usepackage{appendix}
 \newcommand{\R}{\mathbb{R}}
 
 \theoremstyle{plain}
\newtheorem{theorem}{Theorem}

\newtheorem{lemma}[theorem]{Lemma}

\newtheorem{proposition}[theorem]{Proposition}

\theoremstyle{definition}

\newtheorem{definition}[theorem]{Definition}

\newtheorem*{theorem*}{Theorem}
\numberwithin{equation}{section}
\numberwithin{theorem}{section}

\title{Two-phase almost minimizers \\
for a fractional free boundary problem}

\author{Mark Allen}
\address{Mathematics Department, Brigham Young University, Provo, UT 84602}
\email{\textcolor{blue}{allen@mathematics.byu.edu}}

\author{Mariana Smit Vega Garcia}
\address{Department of Mathematics, Western Washington University, Bellingham, WA 98225}
\email{\textcolor{blue}{smitvem@wwu.edu}}
\thanks{M.A. has been partially supported by the Simons Grant 637757 and M.S.V.G has been partially supported by the NSF grant DMS-2054282.}
\date{}

\begin{document}

\begin{abstract}
  In this paper, we study almost minimizers to a fractional Alt-Caffarelli-Friedman type functional. Our main results concern the optimal $C^{0,s}$ regularity of almost minimizers as well as the structure of the free boundary. We first prove that the two free boundaries $F^+(u)=\partial\{u(\cdot,0)>0\}$ and $F^-(u)=\partial\{u(\cdot,0)<0\}$ cannot touch, that is, $F^+(u)\cap F^-(u)=\emptyset$. Lastly, we prove a flatness implies $C^{1,\gamma}$ result for the free boundary.

\end{abstract}
\maketitle
\section{Introduction\label{S:intro}}

\subsection{The Alt-Caffarelli-Friedman Functional}
Alt, Caffarelli and Friedman gave in \cite{ACF} the first rigorous mathematical treatment of the two-phase energy
\begin{equation}\label{ACF}
J(u,U)=\int_{U}|\nabla u|^2+\int_{U}\lambda^+\chi_{\{u>0\}}+\lambda^-\chi_{\{u<0\}},
\end{equation}
where $U\subset \R^n$ is a domain with locally Lipschitz boundary, and $\lambda^{\pm}>0$.
This is a two-phase version of the one-phase free boundary problem (also called the Bernoulli
problem) studied in \cite{AC}, first introduced to model the flow of two liquids in jets and cavities and subsequently found
to have applications to numerous problems including eigenvalue optimization, see \cite{DSV}. 

We say that $u$ is a minimizer of $J$ in the open set $U$ if $J(u,V)\le J(v,V)$ for any open set $V$ with $\overline{V}\subset U$ and any $v\in W^{1,2}(U)$ with $v=u$ in $U\setminus \overline{V}$. Given a minimizer $u$, the free boundaries are defined as $F^{\pm}(u)=U\cap\partial\{\pm u>0\}$. 

When $F^+(u)\cap F^-(u)=\emptyset$, each portion of the free boundary $F^{\pm}(u)$ is the free boundary of a minimizer to a one-phase problem whose regularity was established in \cite{AC}. When $F^{+}(u)=F^{-}(u)$, the regularity of the free boundary was studied in \cite{ACF} when $n=2$, and subsequently in \cite{C}, \cite{C2}, \cite{C3}. More recently, \cite{DSV} addressed the so-called branch points, that is, points around which the free boundary contains both one-phase and two-phase points at every scale. More precisely, let $F_{\text{TP}}(u)=F^{+}(u)\cap F^-(u)$ denote the two-phase points, $F_{\text{OP}}(u)=(F^+(u)\cup F^-(u))\setminus F_{\text{TP}}(u)$ denote one-phase points, and $F_{\text{BP}}(u)=F_{\text{TP}}(u)\cap\overline{F_{\text{OP}}(u)}$ denote the branch points.
In \cite{DSV} (and \cite{SV} when $n=2$), the authors obtain regularity of two-phase points. That is, given any $x_0\in F_{TP}(u)\cap U$, there exists $r_0>0$ such that $F^{\pm}(u)\cap B(x_0,r_0)$ are $C^{1,1/2}$ graphs. This interesting result is particularly significant for branch points, which motivated the question of whether branch points actually occur. The existence of a minimizer for which the free boundary exhibits a branch point was proved in \cite{DESVGT2}, where the authors prove that there exists a minimizer that exhibits a `pool' of zeroes.

Given that branch points can occur, understanding their geometry became an important goal. The authors in \cite{DSV2} proved that for certain symmetric, critical points in dimension two, the branch points in the free boundary are locally isolated. In contrast, when one considers \emph{almost minimizers}, the situation is vastly different: in \cite{DESVGT2}, the authors prove that the set of branch points for almost minimizers can be essentially arbitrary. To describe the almost minimizers that will be addressed in this paper, consider the following notion of almost minimizers for energy functionals, first introduced in \cite{Anz}:
\begin{definition} Let $r_0>0$. A \emph{gauge function} is a function $\omega:(0,r_0)\rightarrow [0,\infty)$ which is non-decreasing with $\omega(0+)=0.$
We say that $u\in W^{1,2}_{\text{loc}}(U)$ is an \emph{almost minimizer} (or $\omega$-minimizer) for the Dirichlet energy $D(u,U)=\int_U|\nabla u|^2$ if for any ball $B_r(x_0)\Subset U$ with $0<r<r_0$, we have \[
D(u,B_r(x_0))\le (1+\omega(r))D(v,B_r(x_0))
\]
for any $v\in u+W^{1,2}_0(B_r(x_0))$. 
\end{definition} 
Almost minimization is the natural property to consider when accounting for the presence of noise or lower-order terms in a problem. 
The articles \cite {DET}, \cite{DESVGT} \cite{DESVGT2}, \cite{DT}, and \cite{SS} studied almost minimizers for functionals of the type first
studied by Alt and Caffarelli in \cite{AC} and Alt, Caffarelli and Friedman in \cite{ACF}. Almost minimizers have also been studied in the context of the Signorini problem in \cite{JP}, in \cite{JPSVG} for the variable coefficient setting, in \cite{JP2} for
the obstacle problem for the fractional Laplacian, and in \cite{JP3} for the parabolic thin obstacle problem. Recently, \cite{DESILVA2022167} and \cite{DSJS} addressed almost minimizers for systems with free boundary.

\subsection{The Fractional Laplacian and Caffarelli-Silvestre Extension}\label{SS:sfractional} \label{s:fraclap}
Throughout the paper we will utilize an extension that allows one to localize the fractional Laplacian which is a nonlocal operator. Given $0<s<1$ and dimension $n$, we denote a point $(x,y)\in\R^{n+1}$ with $x \in \mathbb{R}^n$.  Denote with $B_r(x_0,y_0)=\{(x,y)\in\R^{n+1} \ : \ |(x,y)-(x_0,y_0)|<r\}$. We will refer to the plane $\R^{n}\times\{0\}$ as the thin space, and call $\mathcal{B}_r(x_0):=B_r(x_0,0) \cap \{y=0\}$ the thin ball, while $B_r(x_0,0)\subset \R^{n+1}$ will be the solid ball. When the balls are centered at the origin, we will simply write $\mathcal{B}_r(0)= \mathcal{B}_r$ and $B_r(0,0)=B_r$.

We start by recalling one of the definitions of the fractional Laplacian $(-\Delta_{x})^s$: given $\hat{u}\in L^1(\R^{n},(1+|x|^{n+2s})^{-1})=\{w:\R^n\rightarrow \R \ \text{ measurable }  : \ \int_{\R^n}\frac{|w(x)|}{1+|x|^{n+2s}}dx<\infty\}$,
\[
(-\Delta_{x})^s\hat{u}(x)=C_{n,s}\text{p.v.}\int_{\R^{n}}\frac{\hat{u}(x)-\hat{u}(x+z)}{|z|^{n+2s}}dz,\]
where $C_{n,s}$ is a normalization constant. Nonlocal problems involving the fractional Laplacian have received a surge of attention in recent years. In particular, the extension procedure given in \cite{CS} has led to incredible progress in the field. The localization involves the following Poisson kernel (for the extension operator $L_a$), given by
\[
P(x,y)=C_{n,a}\frac{|y|^{1-a}}{(|x|^2+|y|^2)^{\frac{n+1-a}{2}}}, \qquad (x,y)\in \R^{n}\times\R_+=\R^{n+1}_{+},
\]
where $a=1-2s\in (-1,1)$. One then considers the following convolution:
\[
u(x,y):=\hat{u}\ast P(\cdot,y)=\int_{\R^{n}}\hat{u}(z)P(x-z,y)dz, \quad (x,y)\in \R^{n+1}_+.
\]
Then $u(x,y)$ solves the following Cauchy problem:
\begin{align*}
\begin{cases}
L_au=\text{div}(y^a\nabla u)=0 & \text{ in } \R^{n+1}_+\\
u(x,0)=\hat{u}(x) & \text{ on } \R^{n} ,
\end{cases}
\end{align*}
where $\nabla=\nabla_{x,y}$ is the full gradient. 
For the purpose of this paper, it will be more useful to multiply div$(y^a \nabla )$ by the 
weight $y^{-a}$ to obtain the operator 
\begin{equation}  \label{e:yaoperator}
 \mathcal{L}_a := \Delta +  \frac{a}{y} \partial_y.  
\end{equation}
Clearly, $\text{div}(y^a \nabla w)=0$ if and only if $\mathcal{L}_a w =0.$ The importance is when there is a nontrivial right-hand side and we wish to use the comparison principle, which is best given for non-divergence form equations.

One can recover $(-\Delta_{x})^s\hat{u}$ as the fractional normal derivative on $\R^{n}$:
\[
(-\Delta_{x})^s\hat{u}(x)=-C_{n,a}\lim\limits_{y\rightarrow 0+}y^a\partial_{y}u(x,y), \qquad x\in \R^{n},
\]
understood in the sense of traces. Denoting the even reflection of $u$ in the $y$-variable to all of $\R^{n+1}$ still by $u$, that is,
\[
u(x,y)=u(x,-y), \quad x\in \R^{n}, \ y<0,
\]
then $\hat{u}(x)$ is $s$-fractional harmonic in an open set $U\subset\R^{n}$ if, and only if, $u(x,y)$ satisfies
\begin{equation}\label{La}
\mathcal{L}_au=0 \text{ in } \tilde{U}=\R^{n+1}_+\cup(U\times\{0\})\cup \R^{n+1}_-.
\end{equation}
We have $\mathcal{L}_au=0$ in $\R^{n+1}_{\pm}$, so \eqref{La} is equivalent to 
\[
\mathcal{L}_au=0 \text{ in } B_r(x_0,0)
\]
for any ball $B_r(x_0,0)$ with $x_0 \in U$ such that $B_r(x_0,0)\Subset\tilde{U}$, or, equivalently, $\mathcal{B}_r(x_0)\Subset U$. Noticing that solutions of this equation are minimizers of the weighted Dirichlet energy $\int_{B_r(x_0,0)}|y|^a|\nabla v|^2,$ one obtains Proposition 1.1 from \cite{JP2}, which states that a function $\hat{u}\in L^1(\R^{n},(1+|x|^{n+2s})^{-1})$ is $s$-fractional harmonic in $U$ if, and only if, its even reflected Caffarelli-Silvestre extension $u(x,y)$ is in $W^{1,2}_{\text{loc}}(\tilde{U},|y|^a)$ and for any ball $B_r(x_0,0)$ with $x_0\in U$ and $\mathcal{B}_r(x_0)\Subset U$, we have 
\[
\int_{B_r(x_0,0)}|y|^a|\nabla u|^2\le \int_{B_r(x_0,0)}|y|^a|\nabla v|^2,
\]
for any $v\in u+H_0^1(a,B_r(x_0,0))$ (with $H_0^1(a,B_r(x_0,0))$ defined shortly in the next subsection.)

Motivated by the Alt-Caffarelli-Friedman functional in \eqref{ACF}, and the developments in the theory of the fractional Laplacian, the first named author considered in \cite{A} minimizers of the energy functional
\[
\int_{\Omega} |y|^a|\nabla u|^2+\int_{\Omega \cap \{y=0\}}\lambda^+\chi_{\{u>0\}}+\lambda^-\chi_{\{u<0\}}d\mathcal{H}^{n},
\]
where $\Omega\subset \R^{n+1}$. The main result of \cite{A} states that the two free boundaries $F^+(u)=\partial \{u(\cdot,0)>0\}$ and $F^-(u)=\partial\{u(\cdot,0)<0\}$ cannot touch, that is, $F^+(u)\cap F^-(u)=\emptyset$. While in the case $s=\frac{1}{2}$ the separation of the free boundaries follows from an application of the Alt-Caffarelli-Friedman monotonicity formula, the general case $s\in(0,1)$ is addressed in \cite{A} through the use of a Weiss-type monotonicity formula. The paper \cite{A} followed \cite{AP}, where the case $s=\frac{1}{2}$ was considered, and \cite{CRS}, where the one-phase version for general $s\in (0,1)$ was studied. Regularity of the free boundary in the one-phase problem for $0<s<1$ was given in \cite{DR12}, and when $s=1/2$ the higher regularity of the free boundary was shown in \cite{SS3,DS15}.

\subsection{Almost Minimizers}\label{SS:almost}
Inspired by \cite{A}, and by the potential ways in which almost minimizers might differ from minimizers described in \cite{DESVGT2}, in this paper we are interested in almost minimizers of the energy functional studied in the context of minimizers in \cite{A}
\begin{equation}\label{energy}
J(u,\Omega)=\int_{\Omega}|y|^a|\nabla u|^2+\int_{\Omega \cap \{y=0\}}\lambda^+\chi_{\{u>0\}}+\lambda^-\chi_{\{u<0\}}d\mathcal{H}^{n},
\end{equation}
considered over the class 
\[
H^1(a,\Omega)=\{v\in L_{loc}^1(\Omega) \ : |y|^{\frac{a}{2}}v, \ |y|^{\frac{a}{2}}\nabla v\in L^2(\Omega)\}.
\]
Here $\lambda^{\pm}$ are positive constants. Throughout the paper we will also consider the space $H_0^1(a,\Omega)$ which is the subset of $H^1(a,\Omega)$ with zero trace on $\partial \Omega$. It is well known (see for instance \cite{k97}) that $H^1(a,\Omega)$ and $H_0^1(a,\Omega)$ are the the closures of $C_0^{\infty}(\Omega)$ and $C^{\infty}(\Omega)$ respectively with the norm 
\[
\|v \|_{H^1(a,\Omega)} := \left(\int_{\Omega} (v^2 + |\nabla v |^2) |y|^a \right)^{1/2}
\]
(We remark that sometimes in the literature the $L^2$ norm of $v$ is not weighted for the norm of $H^1(a,\Omega)$; however, standard techniques can show the two spaces are equivalent.)
Following the definition for almost minimizers to the fractional obstacle problem when $0<s<1$ given in \cite{JP2}, we define almost minimizers for the fractional Alt-Caffarelli-Friedman functional when $0<s<1$. 
\begin{definition} Let $r_0>0$ and $\omega:(0,r_0)\rightarrow[0,\infty)$ be a gauge function. We say that $u\in L^1(\R^{n},(1+|x'|^{n+2s})^{-1})$ is an almost minimizer of \eqref{energy} in an open set $U\subset\R^{n}$ with gauge function $\omega$ if its reflected Caffarelli-Silvestre extension $u(x,y)\in W^{1,2}_{\text{loc}}(\tilde{U},|y|^a)$ and for any ball $B_r(x_0,0)$ with $x_0\in U$ and $0<r<r_0$ such that $\mathcal{B}_r(x_0,0)\Subset U$, we have
\[
J(u,B_r(x_0,0))\le(1+\omega(r))J(v,B_r(x_0,0))
\]
for any $v\in u+H_0^1(a,B_r(x_0,0))$.
\end{definition}

We will assume $\omega(r)=\kappa r^{\alpha}$ for some $\kappa>0$ and $\alpha\in(0,1]$.
With the extension and subsequent local nature of an almost minimizer, we will actually consider the following more general definition of an almost minimizer. 

\begin{definition}\label{almostminforus}
   Let $\Omega\subset \R^{n+1}$ be an open set, $u \in H^1(a,\Omega)$ with $u(x,y)=u(x,-y)$ and assume that $u$ is $a$-harmonic in 
   $\Omega \setminus \{y=0\}$. We say that $u$ is an almost minimizer of \eqref{energy} if there exists $\kappa>0$ and $\alpha \in (0,1]$ such that for any ball $B_r(x_0,0) \Subset \Omega$, we have 
\[
J(u,B_r(x_0,0))\le(1+\kappa r^{\alpha})J(v,B_r(x_0,0))
\]
for any $v\in u+H_0^1(a,B_r(x_0,0))$.
\end{definition}
Notice that we make the assumption that almost minimizers are $a$-harmonic off the thin space. Such an assumption is similar to what was assumed for almost minimizers for the fractional obstacle problem in \cite{JP2}. This assumption is natural for $0<s<1$ with $s\neq 1/2$ since the only relevance for $y>0$ comes from the Caffarelli-Silvestre extension for the fractional Laplacian, where all functions and competitors are $a$-harmonic off the thin space. 

\begin{definition}Given $u$ an almost minimizer, we define
\[
F^+(u):=\Omega\cap\partial\{u(\cdot,0)>0\} \qquad 
F^-(u)=\Omega \cap \partial\{u(\cdot,0)<0\}, \qquad F(u)=F^+(u)\cup F^-(u),
\]
and call $F(u)$ the free boundary of $u$.    
\end{definition}

\subsection{Main Results}

Our first main result addresses the optimal regularity of almost minimizers:

\begin{theorem*}[Optimal Regularity, see Theorem \ref{T:s3}] Let $u$ be an almost minimizer of $J$ on $B_1$ with $0<s<1$. Then $u\in C^{0,s}(B_1)$.
\end{theorem*}

Next, we show that, similarly to what happens with minimizers, the two free boundaries $F^+(u)$ and $F^-(u)$ cannot touch:

\begin{theorem*}[Separation of free boundaries, see Theorem \ref{t:sep1}] Let $u$ be an almost minimizer of $J$ on $B_1$ with $0<s<1$. Then $F^+(u)\cap F^-(u)=\emptyset$.
\end{theorem*}

Finally, we prove a flatness implies $C^{1,\gamma}$ result for the free boundary:

\begin{theorem*}[Flatness implies $C^{1,\gamma}$, see Theorem \ref{t:reg}]
 Let $u$ be an almost minimizer to $J$ in $B_1$ with constant $\kappa$ and exponent $\alpha$. Assume that $\|u\|_{C^{0,s}(B_1)} \leq C$ and  $|u-U|\leq \tau_0$ in $B_1$, where in polar coordinates $x_n=\rho\cos(\theta), y=\rho\sin(\theta)$, $U$ (depending only on $x_n,y$) is given as \[
 U(x,y)=\left(\rho^{1/2} \cos(\theta/2)\right)^{2s},
\]
and is the prototypical $2D$ solution for our free boundary problem.  If $\tau_0$ and $\kappa$ are small enough depending on $\alpha,n,s$, then $F(u)$ is $C^{1,\gamma_0}$ in $B_{1/2}$ for some $\gamma_0>0$ depending on $n,\alpha,s$. 
\end{theorem*}

\subsection{Outline of the Paper}

In Section \ref{S:optimal}, we prove optimal $C^{0,s}$ regularity of almost minimizers. In Section \ref{S:nondege}, we discuss nondegeneracy. In Section \ref{S:separation}, we show that the two free boundaries cannot touch, and also show a separation in the full thick space. To prove the regularity of the free boundary we follow the outline in \cite{SStwo} for the case $s=1/2$. In Sections \ref{S:comparison} and \ref{S:compactness} we prove intermediate results necessary for our flatness implies $C^{1,\gamma}$ result for the free boundary, which is done in Section \ref{S:freebdryreg}.

\subsection{Acknowledgments}
The authors would like to thank the anonymous referees who provided useful and detailed comments on an earlier version of the manuscript.

\section{Optimal regularity}\label{S:optimal}

In this section, we prove that almost minimizers have optimal $C^{0,s}$ regularity. The following boundary Harnack principle (Theorem 2.4 in \cite{CRS}) will be utilized throughout the article. 

\begin{proposition} \label{p:bhpa}
 Let $u,v \geq 0$ satisfy $\mathcal{L}_a u=\mathcal{L}_a v =0$ in $B_{2r} \cap \{y>0\}$. Assume  that $u(x,0)=v(x,0)=0$ for $x \in \mathcal{B}_{2r}$. Then there exists a constant $C=C(n,a)$ such that 
 \[
  \sup_{B_r \cap \{y>0\}} \frac{u}{v} \leq C \frac{u(0,r)}{v(0,r)}. 
 \]
 Since $y^{2s}$ is $a$-harmonic, it follows that 
 \[
  \frac{1}{C} y^{2s} \frac{u(0,r)}{r^{2s}} \leq  \sup_{B_r \cap \{y>0\}} u \leq C y^{2s} \frac{u(0,r)}{r^{2s}}.
 \]
\end{proposition}

We start with growth estimates for $a$-harmonic functions. We will utilize the following result from  \cite{JP2}: 

\begin{lemma}[Lemma 3.1 from \cite{JP2}]\label{L:subharmonic} Let $v\in H^1(a,B_R)$ be a solution of $\mathcal{L}_av=0$ in $B_R$. If $v$ is even in $y$, then for $0<\rho<R$,
\[
\int_{B_{\rho}}|\nabla_{x}v|^2|y|^a\le\left(\frac{\rho}{R}\right)^{n+1+a}\int_{B_R}|\nabla_{x}v|^2|y|^a.
\]
\[
\int_{B_{\rho}}|v_{y}|^2|y|^a\le\left(\frac{\rho}{R}\right)^{n+3+a}\int_{B_R}|v_y|^2|y|^a.
\]
\end{lemma}

\begin{theorem}\label{T:s3}[Optimal $C^{0,s}$ regularity]
Let $u$ be an almost minimizer of $J$ on $B_1$ with $0<s<1$. Then $u\in C^{0,s}(B_1)$. 
\end{theorem}

\begin{proof} 
We will follow the outline in \cite{A}. Let $u$ be an almost minimizer in $B_1$. For every $0<r<1$, we consider the $a$-harmonic replacement $v$ of $u$ in $B_r=B_r(x,0)\subset B_1$, that is, $\mathcal{L}_av=0$ in $B_r$ and $v=u$ on $\partial B_r$. Since $u$ is an almost minimizer we obtain 
 \begin{align*}
  \int_{B_r}|y|^a |\nabla u|^2&
  + \int_{\mathcal{B}_r}\lambda^+\chi_{\{u>0\}} + \lambda^-\chi_{\{u<0\}}
 \\
 &\leq 
  (1+\kappa r^{\alpha})\left(\int_{B_r}|y|^a |\nabla v|^2 
  + \int_{\mathcal{B}_r}\lambda^+\chi_{\{v>0\}} + \lambda^-\chi_{\{v<0\}}\right)
 \end{align*}
 so that 
 \[
 \int_{B_r}|y|^a |\nabla u|^2 
 \leq (1+\kappa r^{\alpha})\left(\int_{B_r}|y|^a |\nabla v|^2
 + C r^{n}\right) + Cr^{n}. 
 \]
 We then obtain 
 \[
 \int_{B_r}|y|^a |\nabla u|^2 \leq 
 \int_{B_r}|y|^a |\nabla v|^2 + Cr^{n} + 
 \kappa r^{\alpha} \int_{B_r}|y|^a |\nabla v|^2.
 \]
 Using that $v$ is $a$-harmonic we have that 
 \[
  \int_{B_r} |y|^a |\nabla (v-u)|^2 = 
  \int_{B_r} |y|^a (|\nabla u|^2 - |\nabla v|^2),
 \]
 so that 
 \[
  \int_{B_r} |y|^a |\nabla (v-u)|^2 
  \leq Cr^{n} 
  + \kappa r^{\alpha} \int_{B_r}|y|^a |\nabla v|^2. 
 \]
 We now choose $0<\rho<r<1$ and utilize Lemma \ref{L:subharmonic}
 to obtain 
 \[
 \begin{aligned}
  \int_{B_{\rho}}|y|^a |\nabla u|^2 
  &= \int_{B_{\rho}}|y|^a |\nabla (u-v+v)|^2 \\
  &\leq 2 \left(\int_{B_{\rho}}|y|^a |\nabla (v-u)|^2 
  + |y|^a |\nabla v|^2\right) \\
  &\leq 2 \left(\int_{B_{r}}|y|^a |\nabla (v-u)|^2 
  + \int_{B_{\rho}}|y|^a |\nabla v|^2\right) \\
  &\leq Cr^{n} + \kappa r^{\alpha} \int_{B_r}|y|^a |\nabla v|^2
  + 2\left(\frac{\rho}{r} \right)^{n+1+a} \int_{B_r}|y|^a |\nabla v|^2.\\
  &\leq Cr^{n} + \kappa r^{\alpha} \int_{B_r}|y|^a |\nabla u|^2
  + 2\left(\frac{\rho}{r} \right)^{n+1+a} \int_{B_r}|y|^a |\nabla u|^2.
  \end{aligned}
 \]

 Now choosing $\delta < 1/2$ with the choices $r=\delta^k, \rho=\delta^{k+1},$ and $\mu=\delta^{n}$ we have 
 \[
  \int_{B_{\delta^{k+1}}}|y|^a |\nabla u|^2
  \leq C \mu^k + \left(\kappa \delta^{k \alpha} + C \mu \delta^{1+a} \right)\int_{B_{\delta^{k}}}|y|^a |\nabla u|^2.
 \]
 At this point, fix $\gamma=1/2$ and choose $\delta$ small enough such that $C\mu\delta^{1+a}\le \mu\gamma/2$. Then let $k\ge N\in\mathbb{N}$ be large enough such that $\kappa\delta^{k\alpha}\le\mu\gamma/2$. Then we obtain for $k \geq N$, 
 \[
 \int_{B_{\delta^{k+1}}}|y|^a |\nabla u|^2
  \leq C \mu^k + \mu \gamma \int_{B_{\delta^{k}}}|y|^a |\nabla u|^2.
 \]
 By iterating the above inequality, we obtain that 
 \[
 \begin{aligned}
  \int_{B_{\delta^{k}}}|y|^a |\nabla u|^2
  &\leq C\mu^{k-1} \sum_{j=0}^{k-N} \gamma^j 
   \int_{B_{\delta^{N}}}|y|^a |\nabla u|^2 \\
  &\leq \frac{C}{1-\gamma} \mu^{k-1} \int_{B_{1}}|y|^a |\nabla u|^2\\
  &= \frac{C}{1-\gamma} \frac{(\delta^k)^{n}}{\delta^{n}} \int_{B_{1}}|y|^a |\nabla u|^2.
  \end{aligned}
 \]
 The constant $\delta$ is now fixed from below. Then for any 
 $r=\delta^{k}$ with $k \geq N$, we have 
 \[
 \int_{B_{r}}|y|^a |\nabla u|^2 
 \leq \frac{C}{(1-\gamma)\delta^{n-1}} r^{n} \int_{B_{1}}|y|^a |\nabla u|^2.
 \]
 Then for a new constant $C$ and any $r \leq \delta^N$ we have 
 \[
 \int_{B_{r}(x,0)}|y|^a |\nabla u|^2 
 \leq C r^{n} \int_{B_{1}}|y|^a |\nabla u|^2.
 \]
 From the estimate above the proof follows as in the second half of the proof in Theorem 3.1 in \cite{A} using that $u$ is $a$-harmonic off the thin space. 
\end{proof}

Throughout the paper, we will often make use of the following estimate obtained at the end of the above proof. If $u$ is an almost minimizer in $B_1$, and if $B_r(x,0) \subset B_{3/4}(0)$, then 
\begin{equation} \label{e:sqest}
  \int_{B_r(x,0)} |\nabla u|^2 |y|^a \leq C r^{n},
 \end{equation}
with the constant $C$ depending on dimension $n$, $a$, and $\int_{B_1}|y|^a |\nabla u|^2$. We note that this estimate is obtained only when the ball is centered on the thin space. 

The energy functional $J$ has a natural rescaling. For a fixed point $(x_0,0) \in F(u)$, we will consider the rescaling
\[
u_r(x,y)=\frac{u(rx+x_0,ry)}{r^s}.
\]
The next result shows that by letting $r$ become small, the almost minimizer $u_r$ becomes closer to becoming a minimizer. This technique is a common approach to studying almost minimizers and is why so many results require ``for $r$ small enough.'' 

\begin{proposition} \label{p:rescale}
 Let $u$ be an almost minimizer in $\Omega$ with $B_r(x_0,0)\subset \Omega$. If 
 \[
  u_r(x,y)=\frac{u(rx+x_0,ry)}{r^s},
 \]
 then 
 \[
  J(u_r,B_1)\leq (1+\kappa r^{\alpha})J(v,B_1),
 \]
 for any $v \in u_r + H_0^1(a,B_1)$. 
\end{proposition}

\begin{proof}
 Let $v \in u + H_0^1(a,B_r(x_0,0))$, and denote
 \[
  v_r(x,y)=\frac{v(rx+x_0,ry)}{r^s}.
 \]
 Now 
 \[
  J(u,B_r(x_0,0))=r^n J(u_r,B_1) \quad \text{ and } \quad J(v,B_r(x_0,0))=r^n J(v_r,B_1).
 \]
 Then since 
 \[
 J(u,B_r(x_0,0))\leq (1 + \kappa r^{\alpha}) J(v,B_r(x_0,0)),
 \]
 we have that 
 \[
  J(u_r, B_1) \leq (1 + \kappa r^{\alpha}) J(v_r,B_1). 
 \]
\end{proof}

This next result is an almost-Caccioppoli inequality and will allow us to bound the Dirichlet energy on the interior. 
\begin{proposition} \label{p:caccioppoli}
 Let $u$ be an almost minimizer on $B_2$. Then for any $B_r(x,0) \subset B_1$, there exists two constants $C_1,C_2$ with $C_1$ depending on $n,a$ and $J(u,B_2)$ and $C_2$ depending only on $n,a$ such that 
  \[
   \int_{B_{r/2}(x,0)} |\nabla u|^2 |y|^a \leq C_1 r^{n+\alpha} + \frac{C_2}{r^2} \int_{B_r(x,0)} u^2 |y|^a. 
  \]
\end{proposition}

\begin{proof}
 By rescaling and translation we may assume integration over the unit ball centered at the origin, so by Proposition \ref{p:rescale}
 \[
  J(u,B_1) \leq (1+ \kappa r^{\alpha}) J(v,B_1). 
 \]
 We utilize the standard cut-off function $0 \leq \eta \leq 1$ for the Caccioppoli inequality with $\eta \equiv 1$ on $B_{1/2}$. We let $v=u - \tau \eta^2 u$ in the inequality above with $\tau >0$. We then obtain
 \[
 \begin{aligned}
  \int_{B_1} 2 \tau |\nabla u|^2 \eta^2 |y|^a &\leq \frac{\kappa r^{\alpha}}{1+\kappa r^{\alpha}}  J(u,B_1)  
   - \int_{B_1} 4\tau \eta u \langle \nabla u, \nabla \eta \rangle |y|^a \\
  &\quad + \int_{B_1} 4 \tau^2 \eta^2 u^2 |\nabla \eta|^2 |y|^a 
   - \int_{B_1} 4 \tau^2 \eta^3 \langle \nabla u, \nabla \eta \rangle  |y|^a 
  + \tau^2 \eta^4 |\nabla u|^2 |y|^a. 
\end{aligned}
 \]
 We note that $J(u,B_1)$ is bounded in the rescaling from \eqref{e:sqest}. The final term on the right hand side can be absorbed into the left by choosing $\tau=1/2$ small. The second and third terms can be separated using Cauchy's inequality just as in the standard proof of the Caccioppoli inequality. Consequently, we have 
 \[
  \int_{B_1} |\nabla u|^2 \eta^2 |y|^a \leq C_1 r^{\alpha}+ C_2 \int_{B_1} u^2 |\nabla \eta|^2 |y|^a,   
 \]
 and the conclusion follows. 
\end{proof}

We end this section by showing that blow-up limits of almost minimizers are in fact minimizers of the functional. 

\begin{lemma}\label{L:limits} Let $u_k$ be a sequence of almost minimizers on $B_{r_k M_k}(0,0)$ with $r_k\rightarrow 0$ and $M_k\rightarrow \infty$. Assume that 
\begin{equation} \label{e:assump}
\int_{B_{r_k M_k/2}} |\nabla u_k|^2 |y|^a \leq C (r_k M_k)^n. 
\end{equation}
Then for a subsequence
\[
\tilde{u}_k(x)=\frac{u_k(r_kx)}{r_k^s}
\]
converges uniformly on compact sets to $u$, which is a local minimizer on $\R^{n+1}$ (meaning $u$ is a minimizer on any ball centered on the thin space).
\end{lemma}

\begin{proof}
 By the optimal $C^{0,s}$ regularity for almost minimizers, we have the existence of a subsequence $\tilde{u}_k$ which converges on every compact set in $C^{0,\gamma}$ for every $0<\gamma<s$ to a limit function $u$. For simplicity in notation, we will assume by translation that our ball $B_{R}(x_0,0)$ over which we show $u$ to be a minimizer is centered at the origin $(0,0)$. 

 Fix $R>0$ and $\epsilon>0$, and let $\psi \in H_0^1(a,B_R(0,0))$. We will compare $J(\tilde{u}_k)$ to $J(\tilde{u}_k + \psi + \eta(u -\tilde{u}_k))$ with $\eta$ a cut-off function. Since we are only assuming the weak convergence in $H^1(a,B_R)$, it is possible for $\nabla (u - \tilde{u}_k)$ to concentrate on a sphere; however, the concentration cannot occur on every sphere simultaneously which motivates the following construction. For $j \in \mathbb{N}$ and $0\leq i < j$, let $\eta_{i,j}\geq 0$ be a standard radial cutoff function such that 
 \[
 \begin{cases}
   &\eta_{i,j}(x) =1 \quad \text{ if } |x|\leq R- \epsilon + i\epsilon/j \\
   &\eta_{i,j}(x) =0 \quad \text{ if } |x| \geq R-\epsilon + (i+1)\epsilon/j \\
 \end{cases}
 \]
 and $|\nabla \eta_{i,j}| \leq j/\epsilon$. We define 
 \[
 A_{i,j} :=\{x: R-\epsilon + \epsilon i/j \leq |x| \leq R-\epsilon + \epsilon (i+1)/j\}.
 \]

 From the assumption \eqref{e:assump} and \eqref{e:sqest} we have for all $k$
 \begin{equation} \label{e:pickj1}
  \int_{B_R} (|\nabla \tilde{u}_k|^2 + |\nabla u|^2 + |\nabla \psi|^2)|y|^a \leq C. 
 \end{equation}
 We now choose $j$ large enough such that $C/j < \epsilon$, and note that $j$ is now fixed.
 From \eqref{e:pickj1} for each $k$ there exists some $0\leq i<j$ (with $i$ depending on $k$) such that 
 \begin{equation} \label{e:picki1}
  \int_{A_{i,j}} (|\nabla \tilde{u}_k|^2 + |\nabla u|^2 + |\nabla \psi|^2)|y|^a < \epsilon, 
 \end{equation}
 from which it follows that also 
 \begin{equation} \label{e:picki2}
 \int_{A_{i,j}} |\nabla (\tilde{u}_k-u)|^2 |y|^a < 2\epsilon.  
 \end{equation}
 Recalling that $|\nabla \eta_{i,j}| \leq j/\epsilon$, we utilize the uniform convergence of $\tilde{u}_k$ to $u$ and choose $k \geq N_0 \in \mathbb{N}$ large enough so that $|\tilde{u}_k - u| \leq (\epsilon/j)^2$. We then have 
 \[
  \int_{A_{i,j}} |\nabla (\tilde{u}_k+ \psi + \eta_{i,j}(u - \tilde{u}_k))|^2 |y|^a \leq C \epsilon
 \]
 From the size of the set $A_{i,j}$ it follows that 
 \begin{equation} \label{e:shell}
   J(\tilde{u}_k + \psi + \eta_{i,j}(u -\tilde{u}_k),A_{i,j}) \leq C\epsilon. 
 \end{equation}

 We will now compare $J(\tilde{u}_k,B_R)$ to $J(\tilde{u}_k + \psi + \eta_{i,j}(u -\tilde{u}_k),B_R)$. By Proposition \ref{p:rescale} we have 
 \[
 J(\tilde{u}_k,B_R) \leq (1+\kappa r^{\alpha}) 
 J(\tilde{u}_k + \psi + \eta_{i,j}(u -\tilde{u}_k),B_R). 
 \]
 We observe that 
 \[
 \begin{aligned}
  J(\tilde{u}_k + \psi + \eta_{i,j}(u -\tilde{u}_k),B_R) &= J(u+\psi,B_{R-\epsilon(j-i)/j}) \\
  &\quad + 
  J(\tilde{u}_k + \psi + \eta_{i,j}(u -\tilde{u}_k),A_{i,j}) +
  J(\tilde{u}_k + \psi,B_R \setminus B_{R-\epsilon(j-(i+1))/j}). 
  \end{aligned}
 \]
 %By Proposition \ref{p:rescale} we have 
 %\begin{equation} \label{e:almon}
 %J(\tilde{u}_k,B_R) \leq (1+\kappa r^{\alpha}) 
 %J(\tilde{u}_k + \psi + \eta_{i,j}(u -\tilde{u}_k),B_R). 
 %\end{equation}
 %From the assumption \eqref{e:assump} and \eqref{e:sqest}
 %we may conclude 
 %\begin{equation} \label{e:competitor}
 %J(\tilde{u}_k + \psi + \eta_{i,j}(u -\tilde{u}_k),B_R) \leq C \quad \text{ for all } k. 
 %\end{equation}
 %Now $\tilde{u}_k \rightharpoonup u$ in $H^1(a,B_R)$, but we are not assuming that the %convergence is in $L^2$. Therefore we utilize a cut-off function $\eta_{i,j}$ with $i$ %depending on $k$. For fixed $\epsilon>0$, it follows from \eqref{e:competitor} that there %exists $j \in \mathbb{N}$ such that for any $k$, there exists $i=i(j,k,n,a)$ such that 
 %\begin{equation} \label{e:shell}
 %  J(\tilde{u}_k + \psi + \eta_{i,j}(u -\tilde{u}_k),A_{i,j}) \leq \epsilon 
 %\end{equation}
 %where 
 %\[
 %A_{i,j} :=\{x: R-\epsilon + \epsilon i/j \leq |x| \leq R-\epsilon + \epsilon (i+1)/j\}.
 %\]
 %(Otherwise, $J(\tilde{u}_k + \psi + \eta_{i,j}(u -\tilde{u}_k),B_R)$ would have to be %arbitrarily large.)
 Then cancelling like terms and using the almost monotonicity, we have 
 \[
 \begin{aligned}
  J(\tilde{u}_k,B_{R-\epsilon}) 
  %&\leq \int_{B_{R-\epsilon(j-(i+1))/j}}|\nabla \tilde{u}_k|^2 |y|^a 
 %+\int_{\mathcal{B}_{R-\epsilon(j-(i+1))/j}}\lambda^+\chi_{\{\tilde{u}_k>0\}}+\lambda^-\chi_{\%{\tilde{u}_k<0\}} \\
 &\leq J(\tilde{u}_k, B_{R-\epsilon(j-(i+1))/j})\\
 &\leq \kappa r_k^{\alpha} J(\tilde{u}_k + \psi + \eta_{i,j}(u -\tilde{u}_k),B_R) 
 + J(u+\psi, B_{R-\epsilon(j-i)/j}) \\
 &\quad + J(\tilde{u}_k + \psi + \eta_{i,j}(u -\tilde{u}_k),A_{i,j}) \\
 &\leq C \kappa r_k^{\alpha}  
 + J(u+\psi, B_{R}) + C\epsilon.  
 \end{aligned}
 \]
 By weak convergence in $H^1(a,B_R)$ we have that 
 \[
  \int_{U} |\nabla u|^2 |y|^a \leq \liminf_{k \to \infty} \int_{U}|\nabla \tilde{u}_k|^2 |y|^a,
 \]
 for all $U \subset B_R$. 
 Also, since $\tilde{u}_k \to u$ uniformly on compact sets, we have that 
 \[
  \int_{\mathcal{B}_{R-\epsilon}}\lambda^+\chi_{\{u>0\}}+\lambda^-\chi_{\{u<0\}}
  \leq \liminf_{k \to \infty}
  \int_{\mathcal{B}_{R-\epsilon}}\lambda^+\chi_{\{\tilde{u}_k>0\}}+\lambda^-\chi_{\{\tilde{u}_k<0\}}.
 \]
 Furthermore, $r_k \to 0$ as $k \to \infty$, so putting together the above inequalites we obtain 
 \[
  J(u,B_{R-\epsilon}) \leq  \liminf_{k \to \infty} J(\tilde{u}_k,B_{R-\epsilon}) 
  \leq \liminf_{k \to \infty} C\kappa r_k^{\alpha} 
  + J(u+\psi, B_{R}) + C \epsilon = J(u+\psi,B_R) + C \epsilon.  
 \]
 Since $\epsilon$ can be made arbitrarily small, we conclude 
 \[
 J(u,B_R) \leq J(u+\psi,B_R). 
 \]
\end{proof}

\section{Nondegeneracy}\label{S:nondege}

In this section, we address the nondegeneracy of almost minimizers which will be instrumental in the study of the free boundary. 

\begin{lemma} \label{l:transfer}
 Let $u$ be an almost minimizer on $B_2$. 
 There exists a constant $C$ depending on $n,a,J(u,B_2)$ such that if $u>0$ on $B_r=B_r(x,0)\subset B_1$, and 
 $u \leq \epsilon r^s$ on $\partial B_r(x,0)$, then 
 \[
   \int_{B_{r/2}(x,0)} |\nabla u|^2 |y|^a \leq C \max\{\epsilon^2, r^{\alpha}\} r^{n}. 
 \]
\end{lemma}

\begin{proof}
 From Proposition \ref{p:caccioppoli} we have 
 \[
 \begin{aligned}
 \int_{B_{r/2}(x,0)} |\nabla u|^2 |y|^a &\leq C r^{n+\alpha} + \frac{C}{r^2} \int_{B_r(x,0)} u^2 |y|^a \\
 &\leq C r^{n+\alpha} + C r^{-2} \epsilon^2 r^{2s} r^{n+1+a} \\
 &= C r^{n+\alpha} + C\epsilon^2 r^n \\
 &\leq C \max\{\epsilon^2, r^{\alpha}\}r^n. 
 \end{aligned}
 \]
\end{proof}

The following lemma is a weaker nondegeneracy result. 
\begin{lemma} \label{l:weaker}
Let $u$ be an almost minimizer on $B_1$ with $B_r(e_1r,0) \subset \{u>0\}$. Then there exist constants $\tau,r_0$ depending only on $n,s,\kappa, \alpha, J(u,B_1)$ such that if $r  \leq r_0$, then 
\[
 \sup_{B_{r/2}(r e_1,0)} u \geq \tau r^s. 
\]
\end{lemma}

\begin{proof}
 We apply the usual rescaling of $u(rx)/r^s$, and by Proposition \ref{p:rescale} we may assume we are on $B_1(e_1,0)$. Suppose that 
 \[
  \sup_{B_{1/2}(e_1,0)} u=\epsilon
  \]
  which we will take to be small. 
 Define $\psi(x)$ a cut-off function such that $\psi(x)=0$ for $x \in B_{1/8}(e_1,0)$ and $\psi(x)=1$ for $x \notin B_{1/4}(e_1,0)$. Let $v=\min\{u,\epsilon \psi\}$. We 
 have that $J(u,B_{1/4}(e_1,0))\leq (1+ \kappa r^{\alpha})J(v,B_{1/4}(e_1,0))$. We also have that 
 \[
 \begin{aligned}
 J(u,B_{1/4}(e_1,0)) &= \int_{B_{1/4}(e_1,0)} |\nabla u|^2|y|^a 
 + \int_{\mathcal{B}_{1/4}(e_1)} \lambda^+ \chi_{\{u>0\}} \\
 &= \int_{B_{1/4}(e_1,0)} |\nabla u|^2|y|^a 
 + \lambda^+ |\mathcal{B}_{1/4}(e_1)|,
 \end{aligned}
 \]
 and 
 \[
 \begin{aligned}
 \int_{B_{1/4}(e_1,0)} |\nabla v|^2 |y|^a &= \int_{B_{1/4}(e_1,0) \cap \{u<\psi\}} |\nabla u|^2|y|^a  + \epsilon^2\int_{B_{1/4}(e_1,0) \cap \{u \geq \psi\}} |\nabla \psi|^2|y|^a\\
 \end{aligned}
 \]
 From Lemma \ref{l:transfer} we have 
 \[
  \int_{B_{1/4}(e_1,0)} |\nabla u|^2 |y|^a \leq C \max\{\epsilon^2,r^{\alpha}\},
 \]
 so that 
 \[
  \int_{B_{1/4}(e_1,0)} |\nabla v|^2 |y|^a \leq C \max\{\epsilon^2,r^{\alpha}\}. 
 \]
Then by the almost minimality of $u$ we obtain 
\[
\begin{aligned}
&\int_{B_{1/4}(e_1,0)} |\nabla u|^2|y|^a 
 + \lambda^+ |\mathcal{B}_{1/4}(e_1) | \\
&=J(u,B_{1/4}(e_1,0)) \leq (1+ \kappa r^{\alpha})J(v,B_{1/4}(e_1,0)) \\
&=(1+ \kappa r^{\alpha})\left(\int_{B_{1/4}(e_1,0)} |\nabla v|^2 |y|^a + 
\lambda^+ |\mathcal{B}_{1/4}(e_1)\setminus \mathcal{B}_{1/8}(e_1)| \right)\\
&\leq (1+ \kappa r^{\alpha})\left(C \max\{\epsilon^2,r^{\alpha}\} + \lambda^+|\mathcal{B}_{1/4}(e_1)\setminus \mathcal{B}_{1/8}(e_1)|\right).
\end{aligned}
\]
Thus, 
\[
 |\mathcal{B}_{1/4}(e_1)| \leq (1+\kappa r^{\alpha}) |\mathcal{B}_{1/4}(e_1)\setminus \mathcal{B}_{1/8}(e_1)| + \frac{C}{\lambda^+} \max\{\epsilon^2,r^{\alpha}\}.
\]
We obtain a contradiction if both $r$ and $\epsilon$ are chosen small enough. 
\end{proof}

This next result will transfer this nondegerenate growth to the thin ball. 
\begin{lemma} \label{l:weakthin}
 Let $u$ be an almost minimizer in $B_1$ with $B_r(re_1,0)\subset \{u>0\}$. Then there exist $\tau, r_0>0$ depending only on $n,s,\kappa, \alpha, J(u,B_1),\lambda^+$ such that if $r \leq r_0$, then 
 \[
  \sup_{\mathcal{B}_{r/2}(re_1)} u \geq \tau r^s. 
 \]
\end{lemma}

\begin{proof}
 Suppose by contradiction that the result is not true. Then there exist a sequence of almost minimizers $u_k$ and radii $r_k \to 0$ such that 
 $B_{r_k}(r_k e_1,0)\subset \{u>0\}$ but 
 \[
 \sup_{\mathcal{B}_{r_k/2}(r_ke_1)} u_k \leq r_k^s/k.
 \]
 We rescale by letting $\tilde{u}_k = u_k(r_k x)/r_k^s$. From optimal regularity, we have that $\tilde{u}_k$ is universally bounded and will converge on compact sets of $B_1(e_1,0)$ to $u_0 \geq 0$. Furthermore, $u_0(x,0)=0$ for any $x \in \mathcal{B}_{1/2}(e_1)$. Lemma \ref{l:weaker} allows us to conclude that $u_0$ is not identically zero on $B_{1/2}(e_1,0)$.  Finally, by the almost minimality of $u_k$ we have 
 \[
  \int_{B_{1/2}(e_1,0)} |\nabla \tilde{u}_k|^2 |y|^a 
  \leq (1+ \kappa r_k^{\alpha})\int_{B_{1/2}(e_1,0)}  |\nabla (\tilde{u}_k+\phi)|^2 |y|^a + \kappa r_k^{\alpha} |\mathcal{B}_{1/2}(e_1)|
 \]
 for any smooth $\phi$ that vanishes on $\partial B_{1/2}(e_1,0)$. 
 
 Letting $k \to \infty$, so that $r_k \to 0$, we conclude that $u_0$ minimizes the weighted Dirichlet energy, so that 
 $u_0$ is $a$-harmonic on $B_1(e_1,0)$. Moreover, $u_0$ is not identically zero, but achieves a minimum of zero at an interior point ($u_0(x,0)=0$ for any $x\in \mathcal{B}_{1/2}(e_1)$), thus reaching a contradiction. 
\end{proof}

%In order to obtain a stronger nondegeneracy result, we need the following Lemma, which addresses limits of almost minimizers.

%\begin{lemma}\label{L:limits} Let $u_k$ be a sequence of almost minimizers on $B_{r_k M_k}(0,0)$ with $r_k\rightarrow 0$ and $M_k\rightarrow\infty$. Then for a subsequence
%\[
%\tilde{u}_k(x)=\frac{u_k(r_kx)}{r_k^s}
%\]
%converges uniformly on compact sets to $u_0$, which is a local minimizer on $\R^{n+1}$ (meaning $u_0$ is a minimizer on any ball centered on the thin space).
%\end{lemma}

%\begin{proof}

%Let $x_0\in\R^n$ and  $\psi\in H^1(a,B_r(x_0,0))$. Defining 
%\[
%\psi_k(z)=\psi\left(\frac{z}{r_k}\right),
%\]
%we have for $R<\frac{M_K}{2}$ and $|x_0|<\frac{M_k}{2}$
%\begin{align*}
%    J(\tilde{u}_k,B_R(x_0,0))&=r_k^{n}J(u_k,B_{Rr_k}(r_k x_0,0))\\
%    &\le (1+(Rr_k)^{\alpha})r_k^{n}J(u_k+\psi_k,B_{Rr_k}(r_kx_0,0))\\
%    &=(1+(Rr_k)^{\alpha})J(\tilde{u}_k+\psi,B_R(x_0,0)).
%\end{align*}
%Therefore
%\begin{align*}
%\int_{\mathcal{B}_r(x_0)}\lambda^+\chi_{\{\tilde{u}_k>0\}}+\lambda^-\chi_{\{\tilde{u}_k<0\}}&\le(Rr_k)^{\alpha}\int_{B_R(x_0,0)}|\nabla\tilde{u}_k|^2|y|^a\\
%&\ +(1+(Rr_k)^{\alpha})\Big[\int_{B_R(x_0,0)}2\langle\nabla \tilde{u}_k,\nabla\psi\rangle|y|^a+\int_{B_R(x_0,0)}|\nabla \psi|^2|y|^a\\
%&\ +\int_{\mathcal{B}_R(x_0)}\lambda^+\chi_{\{\tilde{u}_k+\psi>0\}}+\lambda^-\chi_{\{\tilde{u}_k+\psi<0\}}\Big]
%\end{align*}
%Letting $k\rightarrow\infty$ and adding $\int_{B_R(x_0,0)}|\nabla u_0|^2|y|^a$ on both sides, we conclude
%\[
%J(u_0,B_R(x_0,0))\le J(u_0+\psi,B_R(x_0,0))
%\]
%which proves the result. 

%\end{proof}

The next Lemma is the final step before our main result of this section, Theorem \ref{t:thin}.

\begin{lemma} \label{l:increase}
 Let $u$ be an almost minimizer in $B_{rM}(0,0)$ with $u(0,0)=0$. There exist universal $\lambda, M,r_0>0$ such that if $r\leq r_0$ and $B_r(re_1,0)\subset \{u>0\}$, then 
 \[
  \sup_{\mathcal{B}_{rM}(0)} u \geq (1+\lambda) \sup_{\mathcal{B}_{r/2}(re_1)} u. 
 \]
\end{lemma}

\begin{proof}
  Suppose by contradiction that the result is not true. Then there exists a sequence of points $(x_k,0)$ and of almost minimizers $u_k$ on 
 $B_{r_k M_k}$
 with $r_k \to 0$ and 
 \[
 \sup_{\mathcal{B}_{r_k M_k}(0)} u_k =u_k(x_k,0) \leq (1+ \epsilon_k) 
 \sup_{\mathcal{B}_{r_k/2}(r_k e_1)} u_k
 \]
 with $M_k \to \infty$ and $\epsilon_k \to 0$. It is also assumed that $u_k(0,0)=0$. 
 By Lemma \ref{L:limits}, a subsequence of the rescaled functions 
 \[
  \tilde{u}_k(x) := \frac{u_k (r_k x)}{r_k ^s}
 \]
 converges (uniformly on compact subsets) to $\tilde{u}_0$ a local minimizer on $\mathbb{R}^{n+1}$. From the uniform convergence, we conclude that $\tilde{u}_0(0,0)=0$. Moreover, there exists a point $(x_0,0) \in \mathcal{B}_{1/2}(e_1,0)$  such that $\sup_{\mathbb{R}^n} \tilde{u}_0(x,0)= \tilde{u}_0(x_0,0)\geq \tau$, where the last inequality follows from Lemma \ref{l:weakthin}. 
 
 From estimates of the sequence $\tilde{u}_k$ we also obtain 
 \[
  |\tilde{u}_0(x,y)| \leq C |(x,y)|^s \quad \text{ for all } (x,y) \in \mathbb{R}^{n+1}. 
 \]
 Since $\tilde{u}_0$ is a minimizer and not just an almost minimizer, then $\mathcal{L}_a \tilde{u}_0=0$ in $\{u>0\}$. Since $\tilde{u}_0$ is even in the $y$-variable, from the explanation in Section \ref{s:fraclap}, we have that 
 \[
 -c \lim_{y \to 0} y^a \partial_y \tilde{u}_0(x_0,y)=0.
 \]
 Then from the extension principle explained in Section \ref{s:fraclap}, the restriction $\hat{\tilde{u}}_0(x)=\tilde{u}_0(x,0)$ satisfies 
 \[
  (-\Delta)^s \hat{\tilde{u}}_0(x_0) = -c \lim_{y \to 0} y^a \partial_y \tilde{u}_0(x_0,y)=0.
 \]
From the strong maximum principle for the fractional Laplacian, we conclude that $\tilde{u}_0(x)$ is constant everywhere on the connected component of $\{u(\cdot,0)>0\}$ containing $(x_0,0)$, and the constant is greater than or equal to $\tau$. Since $u$ is continuous this contradicts the fact that $\tilde{u}_0(0,0)=0$.  
\end{proof}

We now prove the main result of this section.
\begin{theorem} \label{t:thin}
 Let $u$ be an almost minimizer on $B_1$. If $0 \in F^+(u)$, then there exists a universal $c>0$ depending on $n,s,J(u,B_1)$ such that
 \[
  \sup_{\mathcal{B}_r(0)}u \geq c r^s \ \text{ for } r \leq 1/2. 
 \]
\end{theorem}

\begin{proof}
 We begin by assuming (by rescaling) that $u$ is an almost minimizer on $B_M$ with $M$ as in Lemma \ref{l:increase}. Since the origin is a free boundary point, we choose $x_1 \in \mathcal{B}_{M^{-4}} \cap \{u(\ \cdot \ , 0)>0\}$. Let $r_1 = d(x_1,F(u))$, and let $\tilde{x}_1 \in F(u)$ such that 
 $d(x_1,\tilde{x}_1)=r_1$. By Lemma \ref{l:increase}, there exists $x_2 \in \mathcal{B}_{r_1 M}(\tilde{x}_1)$ with $u(x_2,0) \geq (1+\lambda)u(x_1,0)$.
 We now proceed in the same manner to inductively choose points $x_k$ that satisfy 
 \begin{itemize}
     \item $u(x_{k+1},0)\geq (1+\lambda) u(x_k,0)$ 
     \item if $r_k=d(x_k, F(u))$ and $\tilde{x}_k$ is a free boundary point realising the distance, then $x_{k+1} \in \mathcal{B}_{Mr_k}(\tilde{x}_k)$.
 \end{itemize}
 Notice that $|x_{k+1}-x_{k}|\leq (M+1)r_k$. The induction ends at the index $k_0$ when the sequence $x_k$ leaves $\mathcal{B}_1$, which is guaranteed since the sequence $u(x_k,0)$ grows geometrically and $u$ is bounded on $B_2$. We have 
 \[
  \begin{aligned}
   u(x_{k_0},0)&= \left(\sum_{1< k < k_0} u(x_{k+1},0)-u(x_k,0)\right) + u(x_1)  \\
   &\geq \lambda \sum_{1<k < k_0} u(x_k,0)\geq C \lambda \sum_{1<k < k_0} d(x_k, F(u))^s \geq \frac{C}{(M+1)^s} \lambda \sum_{1<k < k_0} |x_{k+1}-x_{k}|^{s} \\
   &\geq C_1 \lambda \sum_{1<k < k_0} |x_{k+1}-x_{k}| \geq C_1 (|x_{k_0}| -|x_2|)\\
   &\geq C_2 |x_{k_0}|. 
  \end{aligned}
 \]
 This achieves the result for $r=1$. By rescaling we achieve the result for $r \leq 1$. 
\end{proof}

\section{Separation of the free boundaries}\label{S:separation}

In this section, we prove that the two free boundaries of the positive and negative phases for almost minimizers cannot touch, similarly to what happens with minimizers (see \cite{A}). 
\begin{theorem} \label{t:sep1}
 Let $u$ be an almost minimizer in $B_1$. Then $F^+(u) \cap F^-(u)=\emptyset. $
\end{theorem}

\begin{proof}
 Suppose by contradiction that there exists $x_0 \in F^+(u) \cap F^-(u)$. We rescale $u$ at $x_0$ in the following way:
 \[
  u_r(x):= \frac{u(rx+x_0,y)}{r^s}. 
 \]
 By the $C^{0,s}$ regularity, we conclude that for some sequence $r_k \to 0$, that $u_{r_k} \to u_0$ in all compact subsets of $\mathbb{R}^n$. Furthermore, by Lemma \ref{L:limits}, $u_0$ is a minimizer  (not just an almost minimizer) on all balls centered on the thin space. From nondegeneracy in the thin space, Theorem \ref{t:thin},  we obtain that $0 \in F^+(u) \cap F^-(u)$. This contradicts Theorem 6.1 in \cite{A}.  
\end{proof}

We now give a separation in the full thick space of $\mathbb{R}^{n+1}$. 
\begin{theorem} \label{t:sep2}
 Let $u$ be an almost minimizer in $B_1$. Let $0 \in F^+(u) \ (F^-(u))$. Then there exists some $r$  $($depending on $u)$ such that $u(x,y) \geq (\leq) \ 0$ for any $(x,y)$ with $|(x,y)| \leq r$. 
\end{theorem}

\begin{proof}
 We show the case in which $0 \in F^+(u)$. From Theorem \ref{t:sep1} there exists $\rho>0$ such that $u(x,0) \geq 0$ for any $|x| \leq \rho$. Let $v$ satisfy 
 \[
 \begin{cases}
 &\mathcal{L}_a v =0 \quad \text{ in } B_{\rho} \cap \{y>0\} \\
 &v=-u^- \quad \text{ on } \partial B_{\rho} \cap \{y>0\} \\
 &v=0 \quad \text{ on } \mathcal{B}_{\rho}.  
 \end{cases}
 \]
 From the comparison principle for $a$-harmonic functions we have that $v\leq u$ in $B_{\rho} \cap \{y>0\}$. By Proposition \ref{p:bhpa}, we conclude that $-u^- \leq v \leq Cy^{2s}$ on $B_{\rho/2} \cap \{y>0\}$. Then 
 \begin{equation} \label{e:bhpa}
  \frac{u}{y^{2s}} \geq -C \quad \text{ in } B_{\rho/2} \cap \{y>0\}. 
 \end{equation}
 
 Assume by contradiction that there exists $(x_k,y_k) \to (0,0)$ with $u(x_k,y_k)<0$. We let $r_k = d((x_k,y_k),F^+(u))$ with $(\tilde{x}_k,0)\in F^+(u)$ and $r_k=|(x_k,y_k)-(\tilde{x}_k,0)|$. We now rescale by 
 \[
  u_k(x,y)= \frac{u(r_kx +\tilde{x}_k, r_k y)}{r_k^s}. 
 \]
 By optimal regularity, nondegeneracy, Lemma \ref{L:limits}, and Theorem \ref{t:sep1}, we have that for a subsequence $u_k$ converges uniformly on compact sets of $\mathbb{R}^{n+1}$ to a local minimizer $u_0$ with $0\in F^+(u_0)$. From \eqref{e:bhpa} we have that $u_0 \geq 0$. (Note that the negative part of $u$ is bounded by below by a factor of $r_k^{2s}$ where as the blow-up occurs with a scaling of $r_k^s$.)
 By uniform convergence, we conclude that $(x_k,y_k) \to (x_0,y_0)$ with $u(x_0,y_0)=0$ and $d((x_0,y_0),F^+(u_0))=1$. 
 By Theorem I in \cite{A}, $u_0(x,y)>0$ whenever $y>0$, so that necessarily $y_0=0$. From Theorem \ref{t:sep1}, for $k$ large enough 
 \[
  u_k(x,0)\geq 0 \quad \text{ for any } x \in \mathcal{B}_{2}. 
 \]
 Furthermore, since $d((x_k,y_k),F^+(u))=1$ and since $y_k \to 0$, it follows that for $k$ large enough 
 \[
  u_k(x,0)= 0 \quad \text{ for any } x \in \mathcal{B}_{3/4}(x_0), 
 \]
 and the same result will be true for $u_0$. In addition, $u_0 \geq 0$, so by Proposition \ref{p:bhpa} we have the following 
 Hopf-type lemma for $a$-harmonic functions: 
 \begin{equation} \label{e:hopf}
  \lim_{y \to 0^+} y^a \partial_y u_0(x,y)>0 \quad \text{ for any } x \in \mathcal{B}_{3/4}(x_0).
 \end{equation}

 With odd reflection $u_k(x,y)=-u_k(x,-y)$, we have that $u_k$ is $a$-harmonic on the full ball $B_{3/4}(x_0,0)$.  If $w_k(x,y):=y^a \partial_y u_k(x,y)$, and with even reflection $w_k(x,y)=w_k(x,-y)$, then $w_k$ is $-a$-harmonic on the full ball $B_{3/4}(x_0,0)$. 
 By uniform convergence results for $-a$-harmonic functions we have that 
 $w_k(x,y)$ converges uniformly to $y^a \partial_y u_0$ on $\overline{B_{1/2}(x_,0) \cap \{y>0\}}$. Since $u_k(x_k,y_k)<0$ and $u_k(x_k,0)=0$, it follows from the mean value theorem that 
 $y^a \partial_y u_k(x_k,z_k)<0$ for some $0<z_k<y_k$. Then from the uniform convergence for $-a$-harmonic functions we conclude that 
 \[
  \lim_{y \to 0} y^a \partial_y u_0(x_0,0) \leq 0. 
 \] 
 However, this contradicts \eqref{e:hopf}. 
\end{proof}

\begin{proposition} \label{p:newnondegen}
 Let $u\geq 0$ be an almost minimizer on $B_2$. There exists a constant $c$ depending on $n,s,\kappa,\alpha$ and $\int_{B_2} |\nabla u|^2 |y|^a$ 
 such that if $(x,y) \in B_{3/2}$ and if $d((x,y),F(u))\geq d((x,y),\{u=0\})$, then 
 \[
 u(x,y)\geq c d((x,y),F(u))^s. 
 \]
\end{proposition}

\begin{proof}
 Let $r=d((x,y),F(u))$. We need only prove the result for $r$ small. Suppose by contradiction the result is not true. Then there exists $r_k\to 0$, almost minimizers $u_k$, and points $(x_k,y_k)$ satisfying the assumptions such that 
 \[
 \frac{u(x_k,y_k)}{r_k^s} \to 0. 
 \]
 Let $(\tilde{x}_k,0) \in F(u_k)$ with $r_k=d((x_k,y_k),F(u_k))$.
 By translation, we assume $(\tilde{x}_k,0)=(0,0)$.
 Define 
 \[
  \tilde{u}_k(x,y):=\frac{u(r_k x,r_k y)}{r_k^s}.
 \]
  Then there exists a subsequence such that $\tilde{u}_k \rightharpoonup u_0$ in $H^1(a,B_2)$ and uniformly in $B_2$. Furthermore, $u_0$ is a minimizer in $B_2$ and by nondegeneracy is not identically zero. Finally, $(x_k/r_k,y_k/r_k) \to (x_0,y_0) \in \partial B_1$ and $d((x_k/r_k,y_k r_k),F(\tilde{u}_k))=1$, and $d((x_k/r_k, y_k/r_k),\{\tilde{u}_k=0\})\geq 1/4$. From the uniform limit, we have that $u(x_0,y_0)=0$. Now $u>0$ whenever $|y|>0$, so by the strong minimum principle for $a$-harmonic functions, we conclude $y_0=0$. Now from Lemma \ref{l:weakthin}, we have for large enough $k$ that 
  \[
   \sup_{\mathcal{B}_{1/2}(x_0)} \tilde{u}_k \geq c. 
  \]
  Furthermore, for large enough $k$ as $y_k \to 0$, we have 
  $u(x,0)>0$ for $x \in B_{3/4}(x_0)$, so that 
  \[
   \int_{B_{3/4}(x_0,0)} |\nabla \tilde{u}_k|^2 |y|^a 
   \leq (1+ \kappa r_k^{\alpha}) \int_{B_{3/4}(x_0,0)} |\nabla (\tilde{u}_k + \psi)|^2 + r_k^{\alpha} |\mathcal{B}_{3/4}|.
  \]
  Then in the limit, we conclude $u_0$ is $a$-harmonic, nonnegative, and not identically zero, so that by the strong minimum principle for $a$-harmonic functions $u>0$ in $B_{3/4}(x_0,0)$, but this contradicts $u(x_0,0)=0$.  
\end{proof}
%\begin{remark} \label{r:infnondegen}
% Using a compactness argument similar to that as given above, %combined with Lemma \ref{l:weaker} and the strong minimum %principle for $a$-harmonic functions, we can conclude that if %$B_r(e_1 r,0) \subset \{u>0\}$, then 
% \[
%  \inf_{B_{r/2}(re_1,0)} u \geq c r^s. 
% \]
%\end{remark}

\section{Comparison subsolutions and supersolutions}\label{S:comparison}
Our remaining goal for the paper is to prove a flatness implies $C^{1,\gamma}$ result for the free boundary, which will be done in Section \ref{S:freebdryreg}. This will be done following the ideas from both \cite{SS3} and \cite{SStwo} (where the regularity of the free boundary is studied for the case $s=1/2$).

 Having proven that $F^+(u) \cap F^-(u)=\emptyset$ and that $u \geq  (\leq) 0$ in a neighborhood of any point of $F^+(u)$ (or $F^-(u))$), the study of the free boundary is reduced to the study of the one-phase problem. Consequently, for the rest of the paper we will assume that $u \geq 0$, so that we may assume $\lambda^-=0$. Furthermore, by a multiplying factor we may assume that $\lambda^+$ is chosen appropriately so that $U$ (defined below) and not a constant multiple of $U$ is a solution. 
 
 We first recall from Proposition \ref{p:rescale} that the standard rescaling 
\[
u_r(x):=\frac{u(rx)}{r^s}
\]
of an almost minimizer $u$
satisfies $J(u_r,B_1) \leq (1+\kappa r^{\alpha}) J(v_r,B_1)$. This can be rewritten as $J(u_r,B_1) \leq (1+\sigma) J(v_r,B_1)$. Assuming that $\|u\|_{H^{1}(a,B_1)} \leq C$, and choosing our test function $v_r$ to be the minimizer of $J$ on $B_1$ with $v_r=u_r$ on $\partial B_1$, we conclude that 
\[
 J(u_r, B_1) \leq J(v_r,B_1) + C \sigma. 
\]
As in \cite{SStwo}, we will utilize the above inequality to consider the regularity of the free boundary. Consequently, for the rest of the paper we assume that $u$ satisfies the following three properties in $B_1$. 
\begin{itemize}
\item $(H1)$ \ $u$ is even in $y$ and $\|u\|_{C^{0,s}(B_1)}, \|u \|_{H^1(a,B_1)} \leq C$. 
\item $(H2)$ \ $\displaystyle J(u,B_1) \leq J(v,B_1) + \sigma$ for some small $\sigma$. 
\item $(H3)$ \ $\displaystyle u(x,y) \geq c \frac{d((x,y),\{u=0\})^{2s}}{d((x,y),F(u))^s}$. 
\end{itemize}
In \cite{SStwo} four conditions are assumed; however, since we make the assumption that almost minimizers are $a$-harmonic off the thin space, only three conditions are needed. Recall the discussion immediately after Definition \eqref{almostminforus} for why this is a natural assumption.

 We now show why condition $(H3)$ is satisfied for minimizers; this is a consequence of nondegeneracy and the boundary Harnack principle. 

 \begin{proposition} \label{p:h3}
  Let $u$ be an almost minimizer on $B_2$. Then $u$ satisfies condition $(H3)$ on $B_1$. 
 \end{proposition}

 \begin{proof}
     Let $(x,y) \in B_1$ and let $\overline{d}=d(x,F(u))$ with $(x_0,0) \in F(u)$
     and $d((x_0,0),(x,y))=\overline{d}$. If  $d((x,y),\{u=0\}) \geq \overline{d}/4$, then $(H3)$ is a consequence of the nondegeneracy result in Proposition \ref{p:newnondegen}.

      We now assume the case in which $d((x,y),\{u=0\}) \leq \overline{d}/2$. From Proposition \ref{p:newnondegen} we have $u(x,\overline{d}/2) \geq c_3 d^s$.   From Proposition \ref{p:bhpa} we have 
     \[
     u(x,y) \geq c u(x,\overline{d}/2) \frac{y^{2s}}{\overline{d}^{2s}} \geq c \frac{y^{2s}}{\overline{d}^s}
     \geq c \frac{d((x,y),\{u=0\})^{2s}}{d((x,y),F(u))^s}. 
     \]
 \end{proof}

 We start by recalling (see \cite{SSS}) the prototypical $2D$ minimizer for our free boundary problem (which up to multiplicative constant) is given by 
\[
 U(\rho,\theta)=\left(\rho^{1/2} \cos(\theta/2)\right)^{2s}. 
\]
It will also be convenient to reference 
\[
 H(\rho,\theta)=\rho^{1/2} \cos(\theta/2). 
\]

We recall the following properties from \cite{SSS}. If $\tau= \rho \cos \theta, \quad \eta=\rho \sin \theta, \quad \rho \geq 0$, $- \pi \leq \theta \leq \pi,$ then 
\begin{equation} \label{e:U}
 \frac{U_{\tau}}{U} = \frac{U_\rho}{U}=\frac{s}{\rho}. 
\end{equation}
From equation (3.1) in \cite{SStwo} we have 
\[
\frac{H_{\tau}(\tau_1,\eta)}{H_{\tau}(\tau_2,\eta)} \leq C \quad \text{ if } |\tau_1-\tau_2|\leq \frac{1}{2}|(\tau_2,\eta)|,
\]
from which it follows (with a new constant $2C$) that 
\begin{equation} \label{e:diadic}
 \frac{U_\tau(\tau_1,\eta)}{U_\tau(\tau_2,\eta)} \leq C \quad \text{ if } |\tau_1-\tau_2|\leq \frac{1}{2}|(\tau_2,\eta)|. 
\end{equation}
We now recall the family of functions $V_{S,\zeta}$ from \cite{SStwo} (see also \cite{SS} for a more detailed explanation). Let us motivate the construction of $V_{S,\zeta}$. The idea is to modify the prototypical $2D$ minimizer to construct a subsolution or supersolution to our free boundary problem. In the simple case of the original Alt-Caffarelli problem, the prototypical $1D$ minimizer would be $x_1^+$. Now taking a hyper-surface $S$, if $\tau$ is the distance to $S$, then 
\[
\Delta_x \tau = - \kappa(x)
\]
where $\kappa$ is the mean curvature of the level set of the distance function $\tau$ to $S$. Then by choosing $S$ appropriately we can have either $\Delta \tau \geq (\leq ) 0$. Furthermore, we have that $|\nabla \tau|=1$ on $S$, so the approach to the boundary remains the same. Thus, $\tau$ would be either a subsolution or supersolution to the original Alt-Caffarelli problem. For the thin one-phase problem, the construction is more difficult since the prototypical minimizer is $2D$. 

For any $\zeta \in \mathbb{R}$ we initially define the following family of functions 
\begin{equation} \label{e:v}
 v_{\zeta}(\tau,\eta):= \left(1+ \frac{\zeta}{4} \rho \right)U(\tau,\eta). 
\end{equation}

For an $(n-1)$-dimensional $C^2$ surface $S\subset \mathbb{R}^{n}$ and a point $X=(x,\eta)$ in a small neighborhood of $S$, we call $P_{S,X}$ the $2D$ plane passing through $X$ and parallel to both the normal to $S$ and the vector $(0,\eta)$. We then define 
\begin{equation} \label{e:familyv}
 V_{S,\zeta}(X):=v_\zeta(\tau,\eta).
\end{equation}
 If $(x_1, \ldots, x_{n-1})=x'$ and 
 \[
 S:= \left\{x_{n}=\xi' \cdot x' + \frac{1}{2} (x')^T M x'\right\}, 
 \]
for some $\xi'$ and $M$, then we define $V_{M,\xi',\zeta}(X):=V_{S,\zeta}(X)$. Finally, for small $\mu>0$ we denote the class
\[
\mathcal{V}_{\mu}:=\{V_{M,\xi',\zeta} : \ \| M\|, |\zeta|, |\xi'| \leq \mu\}. 
\]

The next lemmas construct a viscosity subsolution and supersolution to our free boundary problem. We recall the notion $\partial V / \partial \tau^s=1$ as given in \cite{SSS} and \cite{SStwo} as 
\[
V(x,y)=U((x-x_0) \cdot \nu(x_0),z) +o(|(x-x_0,z)|^{s}), \quad \text{ as } (x,y) \to (x_0,0),
\]
where $\nu(x_0)$ denotes the unit normal at $x_0$ to $\partial \{V(\cdot,0)>0\}$. 
\begin{lemma} \label{l:viscsub}
 Let $V=V_{M,\xi',\zeta} \in \mathcal{V}_{\mu}$ with $\mu \leq \mu_0$. There exists $C_0>0$ such that if
 \[
 \frac{\zeta}{1-a} - \text{\emph{tr}} \ M \geq C_0 \mu^2, 
 \]
 then $V$ is a comparison subsolution to the thin one-phase problem in $B_2$:
 \[
 \begin{aligned}
  &(i) \quad \mathcal{L}_a V\geq \mu^2 |y| \text{ in } B_2^+(V), \\
  &(ii) \quad \frac{\partial V}{\partial \tau^{s}}=1 \text{ on } F(V). 
 \end{aligned}
 \]
 Similarly, if 
  \[
 \frac{\zeta}{1-a} - \text{\emph{tr}} \ M \leq -C_0 \mu^2, 
 \]
 then $V$ is a comparison supersolution to the thin one-phase problem in $B_2$: 
  \[
 \begin{aligned}
  &(i) \quad \mathcal{L}_a V \leq -\mu^2 |y| \text{ in } B_2^+(V), \\
  &(ii) \quad \frac{\partial V}{\partial \tau^{s}}=1 \text{ on } F(V). 
 \end{aligned}
 \]
\end{lemma}

\begin{proof}
  From the definition of $V$ the free boundary condition $(ii)$ is satisfied. We therefore only need to check condition $(i)$. Since the variable $\tau$ is the distance (in the $x$ variable) to the surface $S$, we have (under a rotation in $x$) that $\partial_{x_{n}} \tau =1$, so that $\partial_{x_{n}}^2 \tau=0$. Then 
 \[
 \Delta_{x} \tau = - \kappa(x),
 \]
 where $\kappa(x)$ is the mean curvature of the level set of the distance function $\tau$ (in the $x$ variable) to $S$. As in the proof of Proposition 3.2 in \cite{SS3}, the principle curvatures are given by 
 \[
 \kappa_i(x) = \frac{\kappa_i(x^*)}{1-\tau \kappa_i(x^*)}, 
 \]
 where $x^*$ is the closest point on $S$ to $x$. As shown in \cite{SS3}, we have 
 \begin{equation} \label{e:curv}
 \begin{aligned}
   |\kappa(x)| &\leq C \mu, \\
   |\kappa(x)-\text{tr } M| &\leq C \mu^2. 
 \end{aligned}
 \end{equation}
 We now compute 
 \[
  \mathcal{L}_a V(x,y) = \left(\frac{\partial^2 }{\partial \tau^2} + \frac{a}{\eta} \frac{\partial}{\partial \eta} + \frac{\partial^2 }{\partial \eta^2} \right)v_{\zeta} - (\partial_\tau v_{\zeta}) \kappa(x). 
 \] 
 Using polar coordinates we have 
 \[
 \begin{aligned}
  \mathcal{L}_a v_{\zeta} &= (1+(\zeta/4) \rho)\mathcal{L}_a U + (\zeta/4)(\mathcal{L}_a \rho) U +  2\langle \nabla (1+\zeta/4 \rho), \nabla U \rangle \\
  &= 0 + (\zeta/4) \frac{1+a}{\rho} U + (\zeta/2) \frac{\partial}{\partial \rho} U \\
  &= (\zeta/4)\frac{1+a}{s} U_{\rho} +(\zeta/2) U_{\rho} \\
  &= \frac{\zeta}{1-a} U_{\tau}.
  \end{aligned}
 \]
Moreover,
 \[
 \begin{aligned}
  \partial_\tau v_{\zeta} - U_\tau &= (\zeta/4) (\partial_\tau \rho) U + (\zeta/4)\rho U_\tau \\
  &=(\zeta/4) (\cos \theta) U + (\zeta/4) s U. 
 \end{aligned}
 \]
 Since $\rho \leq 2$, then from \eqref{e:U} it follows that $U \leq  U_{\tau}/s$, so that
 \begin{equation} \label{e:diffe}
 |\partial_t v_{\zeta} - U_{\tau}| \leq |\zeta/2| U \leq  \frac{\mu}{s} U_{\tau}. 
 \end{equation}
 Therefore 
  \[
   \begin{aligned}
     \mathcal{L}_a V(x,y) -(\zeta/(1-a) - \text{tr} M) U_{\tau} 
     &=  \frac{\zeta}{1-a} U_{\tau} - \partial_\tau v_{\zeta} \kappa(x) -(\zeta/(1-a) - \text{tr} M) U_{\tau}  \\
     &= - \partial_\tau v_\zeta \kappa(x) + \text{tr} M  \ U_{\tau},
   \end{aligned}
  \]
  and so from \eqref{e:curv} and \eqref{e:diffe} we conclude 
  \[
   \begin{aligned}
       |\mathcal{L}_a V(x,y) -(\zeta/(1-a) - \text{tr} M) U_{\tau}| &= 
       | (- \partial_\tau v_\zeta + U_{\tau}) \kappa(x) +  (\text{tr} M -\kappa(x)) \ U_{\tau}| \\
       &\leq \frac{\mu}{s} U_{\tau} |\kappa(x)| + C \mu^2 U_{\tau} \\
       &\leq C \mu^2 U_{\tau}. 
   \end{aligned}
  \]
  We may then conclude that if $\zeta/(1-a) - \text{tr}M \geq 2C \mu^2$, then 
  \[
    \mathcal{L}_a V(x,y) \geq C \mu^2 U_{\tau} \geq C_0 \mu^2 |y|^{1-a}. 
  \]
  Similarly, if $\zeta/(1-a) - \text{tr}M \leq -2C \delta^2$, then 
  \[
    \mathcal{L}_a V(x,y) \leq -C \mu^2 U_{\tau} \leq -C_0 \mu^2 |y|^{1-a}. 
  \]
\end{proof}

As in \cite{SS3} we utilize the domain deformation given by 
\[
 U(x,y)=V(x-\tilde{V}(x,y)e_n,y).
\]
We have the following result for $\tilde{V}$. 

\begin{proposition} \label{p:hodograph}
 Let $V=V_{M,\xi',\zeta} \in V_{\mu}$ with $\mu \leq \mu_0$. Then $V$ is strictly monotone increasing in the $e_n$-direction in $B_2^+(V)$. Moreover, $\tilde{V}$ satisfies the following estimate in $B_2$
 \[
  |\tilde{V}(x,y) - \gamma_V(x,y)| \leq C_1 \mu^2, \quad \gamma_V(x,y)= - \xi' \cdot x' - \frac{1}{2}(x')^T M x' + \frac{\zeta}{4s}(x_n^2+y^2). 
 \]
\end{proposition}

\begin{proof}
 We first show that $v_{\zeta}$ satisfies 
 \begin{equation} \label{e:trap}
   U(\tau + \frac{\zeta}{4s} \rho^2 -C\mu^2,\eta) \leq v_{\zeta}(\tau, \eta) \leq 
   U(\tau + \frac{\zeta}{4s} \rho^2 +C\mu^2,\eta),
 \end{equation}
 with $\rho^2=\tau^2 + \eta^2$ and $\gamma_{\zeta}(\tau,\eta):=\frac{\zeta}{2}\rho^2$. Taking derivatives of \eqref{e:U} we obtain
 \[
 |U_{\tau \tau}|\leq \frac{2}{\rho} U_{\tau}. 
 \]
If $\delta \leq \rho/2$, then 
\[
|U(\tau + \delta, \eta)-\left(U(\tau, \delta) + \mu U_{\tau}(\tau, \delta) \right)|
\leq \delta^2 |U_{\tau \tau}(\tau', \eta)| \leq C \delta^2 \rho^{-1} U_{\tau}(\tau,\eta). 
\]
The last inequality utilizes \eqref{e:diadic}. Using again \eqref{e:U} we then have 
\[
 \left(1+ \frac{s}{\rho}\delta + C \frac{s}{\rho^2} \delta^2 \right)U(\tau, \eta) 
 \geq U(\tau + \delta,\eta) 
 \geq \left(1+ \frac{s}{\rho}\delta - C \frac{s}{\rho^2}\delta^2 \right)U(\tau, \eta).  
\]
We now choose $\delta$ to satisfy 
\[
 \frac{s}{\rho} \delta + C \frac{s}{\rho^2} \delta^2 = \frac{\zeta}{4}\rho,
\]
so that 
\[
 \left(1+ \frac{\zeta}{4} \right)U(\tau, \eta) \geq 
 U\left(\tau + \frac{1}{s}\frac{\zeta}{4} \rho^2 - C \frac{\delta^2}{\rho}\right). 
\]
From the quadratic equation, we have that $\delta/\rho = O(\zeta)$. Since $|\zeta| \leq \mu$ we have that 
\[
 \left(1+ \frac{\zeta}{4} \right)U(\tau, \eta) \geq 
 U\left(\tau + \frac{1}{s}\frac{\zeta}{4} \rho^2 - C \mu^2,\eta \right). 
\]
Applying the same computations to 
\[
 \frac{s}{\rho} \delta - C \frac{s}{\rho^2} \delta^2 = \frac{\zeta}{4}\rho,
\]
 we obtain \eqref{e:trap}. 

To finish the proof of the proposition, we need to relate $\tau$ to $x_n$. Since $\tau$ is the signed distance to the surface $S:= \{x_n = g(x')\}$, as shown in (3.14) in \cite{SS3}, we have 
that 
\begin{equation} \label{e:distbound}
 1 \geq \frac{\partial \tau}{\partial x_n} \geq 1 - C \mu^2 \quad \text{ in } B_1. 
\end{equation}
By integrating this inequality along the segment from $(x',g(x'))$ to $(x',x_n)$, we obtain 
\[
|\tau - (x_n - g(x'))| \leq C \mu^2. 
\]
Furthermore, in $B_1$ the surface $S$ and $\{x_n=0\}$ are within distance $\mu$, so that $|\tau - x_n|\leq C \mu$. Combining with  $|\zeta| \leq \mu$, we conclude 
\[
 \left| \frac{\zeta}{4s}(\tau^2 + \eta^2) - \frac{\zeta}{4s} (x_n^2 + \eta^2) \right| 
 \leq \frac{|\zeta|}{4s} |\tau + x_n||\tau - x_n| \leq C \mu^2. 
\]
Applying the above inequalities to \eqref{e:trap} we conclude 
\begin{equation} \label{e:trap1}
\begin{aligned}
    &U\left(x + \left(\frac{\zeta}{4s}(x_n^2 + y^2) -\frac{1}{2}(x')^T M x' - \xi'\cdot x' -C\mu^2\right)e_n,y\right) \\
    &\leq v_{\zeta}(x, y)  \\
    &\leq U\left(x + \left(\frac{\zeta}{4s}(x_n^2 + y^2) -\frac{1}{2}(x')^T M x' - \xi'\cdot x' +C\mu^2\right)e_n,y\right). \\
   \end{aligned}
\end{equation}
\end{proof}

%Although we do not directly use the following result, the proof is illustrative and will be used again later.  
\begin{lemma} \label{l:n-1direction}
 Let $V \in \mathcal{V}_{\mu}$, and let $S=\{x_n = g(x')\}$. If $\mu \leq \mu_0$, then there exist constants $c,C$ such that 
 \[
  c U_{\tau}((x_n - g(x'))/2,y) \leq V_{x_n}(x-g(x')e_n,y) \leq C U_{\tau}(x_n - g(x'),y). 
 \]
\end{lemma}

\begin{proof}
 Note that 
 \[
  V_{x_n} = \partial_{\tau} v_{\zeta}(\tau,\eta) \frac{\partial \tau}{\partial x_n}. 
 \]
 Recalling \eqref{e:distbound} we have 
 \[
  \frac{1}{2} \partial_{\tau} v_{\zeta} \leq V_{x_n} \leq \partial_{\tau} v_{\zeta}.
 \]
 Since 
 \[
 \begin{aligned}
      \partial_{\tau} v_{\zeta} &= U_{\tau} + \frac{\zeta}{4}\rho U_{\tau}
      + \frac{\zeta}{4} \cos \theta \ U \\
      &= U_{\tau} + \frac{\zeta}{4}(s+ \cos\theta) U,
 \end{aligned} 
 \]
 we have $(1/2) U_{\tau} \leq \partial_{\tau} v_{\zeta} \leq 2 U_{\tau}$. Integrating along the line segment from $(x',g(x'),y)$ to $(x',x_n,y)$ and using \eqref{e:distbound}, we obtain 
 \[
 \frac{1}{2}(x_n - g(x')) \leq \tau \leq x_n - g(x'). 
 \]
 The conclusion then follows. 
\end{proof}

The biggest difficulty in studying almost minimizers is that unlike minimizers, they do not satisfy an Euler-Lagrange equation. Instead, one must utilize the energy functional directly. Since almost minimizers do not satisfy a PDE in the positivity set, we cannot utilize a comparison principle. This next Lemma is a key element that will establish a comparison principle for almost minimizers. 
\begin{lemma} \label{l:subsol}
 Assume $u$ satisfies $(H1)$-$(H3)$, and $u \geq V=V_{M,\xi',\zeta} \in \mathcal{V}_{\mu^{2/3}}$
 in $B_1$. If  
 \[
  \frac{\zeta}{1-a} - \text{\emph{tr}}M \geq \mu, 
 \]
 and $\sigma \leq \mu^C$ with $C=C(s,n)$ large, then 
 \[
 u(x,y) \geq V_{M,\xi',\zeta}(x+t_0 e_n,y) \quad \text{ in } B_{1/4} \text{ for some } t_0>0. 
\]
\end{lemma}

\begin{proof}
 We will show that for each point $(x_0,y_0) \in \overline{B}_{1/4}$ there exist some $r$ and $t$ both depending on $(x_0,y_0)$ such that $V(x+t,y) \leq u$ for $(x,y) \in B_{r}(x_0,y_0)$. Since $\overline{B}_{1/4}$ is compact, we may choose a finite covering and find a minimum $t_0$, and the conclusion of the Lemma follows. 

 We first consider the case in which $y_0>0$. 
 Since $\mathcal{L}_a(u-V) \leq 0$ whenever $y>0$, it follows from the comparison principle for $\mathcal{L}_a$ that $(u- V)>0$ in a neighborhood of $(x_0,y_0)$, and the existence of $t,r$ depending on $(x_0,y_0)$ easily follows. 
 
 We now consider the case in which $y_0=0$. We first treat the subcase in which $x_0$ is an interior point for the coincidence set (i.e. $u(x,0)=0$ if $|x-x_0|<r_1$ is small enough). 
 Then from the Hopf principle \eqref{e:hopf} for the equation $\mathcal{L}_a$ we have 
 \[
  \lim_{y \to 0} y^a \partial_y (u - V)(x_0,y)>0. 
 \]
 Also 
 \[
  \lim_{y \to 0} y^a \partial_y (u - V)(x_0,y)=L \quad \text{ with } L < \infty.  
 \]
 (This last statement is true because $(x_0,0)$ is an interior point for the coincidence set.)
 Then there exists $L < L_1 < \infty$ as well as $r_2$ small enough such that 
 \[
  V(x,y) < L_1 y^{1-a} < u(x,y) \quad \text{ for any } (x,y) \in B_{r_2}(x_0,0).  
 \]
 It follows that for small $t$ that $V(x_0 + te_n,y) \leq u(x,y)$ in $B_{r_2/2}(x_0,0)$.  

 We now turn to the more difficult case in which $(x_0,0) \in \overline{\{V(x,0)>0\}}$. We will now
 utilize the fact that $u$ is an almost minimizer (property $H1$). 
 We define 
 \[
 \overline{V}(x,y) := V_{\overline{M}, \xi', \overline{\zeta}}\left(x + \frac{\mu}{8n} e_n,y\right), \quad \overline{\zeta}= \zeta - \frac{\mu}{2n}, \quad \overline{M}= M + \frac{\mu}{2n} I. 
 \]
 From Proposition \ref{p:hodograph}, we have 
 \[
 \overline{V}(x,y)\leq V(x,y) \quad \text{ near } \partial B_1, \quad \overline{V}(x,y)\geq V\left(x+ \frac{\mu}{16n},y\right) \quad \text{ in } B_{1/2}. 
 \]

We now set 
\[
 u_{\max} := \max\{u, \overline{V}\}, \quad u_{\min}:=\min\{u,\overline{V}\},
\]
and note that 
\[
 u_{\max} = u, \quad u_{\min} = \overline{V} \quad \text{ near } \partial B_1. 
\]
Using that $u$ is an almost minimizer we have 
\[
J(u,B_1) \leq J(u_{\max},B_1) + \sigma. 
\]
From direct computations we have 
\[
J(u_{\max},B_1) + J(u_{\min},B_1)=J(\overline{V},B_1) + J(u,B_1). 
\]
Putting the above two together we have 
\begin{equation} \label{e:sigabove}
J(u_{\min},B_1) - J(\overline{V}, B_1) \leq \sigma.  
\end{equation}
In order to obtain an inequality from below for the difference in energy we first consider the modified functional 
\[
\tilde{J}(w,B_1) = J(w,B_1) + \mu^2 \int_{B_1} 2w |y|,
\]
among all $H^1(a,B_1)$ functions $w$ satisfying $\min\{V,\overline{V}\} \leq w \leq \overline{V}$. From Proposition \ref{p:viscmin} we have that $\overline{V}$ is the minimizer, and thus we can conclude 
\begin{equation} \label{e:Ebelow}
  J(u_{\min},B_1) - J(\overline{V},B_1) \geq 2\mu^2 \int_{B_1} (\overline{V}- u_{\min})|y|. 
\end{equation}

Now we will assume by contradiction that $V$ is tangent by below to $u$ at some point $(x_0,0) \in \overline{\{V(x,0)>0\}} \cap \overline{B}_{1/4}$. This means that any translation $V(x+t e_n,y)$ with $t>0$ cannot be below $u$ in a neighborhood of $(x_0,0)$. We now aim to show an integral bound of the form 
\begin{equation}  \label{e:Ebelow1}
 \int_{B_1} \left(\overline{V} - u_{\min}\right) |y| \geq \mu^C. 
\end{equation}
To do so we utilize $\overline{V}$ which satisfies $\overline{V}(x,y)\geq V(x + \mu/(16n) e_n,y)$ in 
$\overline{B}_{1/4}$. Then 
\[
 (\overline{V} - u)(x_0,0) \geq V(x_0 + \mu/(16n) e_n,0) - u(x_0,0)= V(x_0 + \mu/(16n) e_n,0) - V(x_0,0). 
\]
%Now 
%\[
%V_{x_n} = \partial_{\tau} v_{\zeta}(\tau,\eta) \frac{\partial %\tau}%{\partial x_n}.
%\]
%Using \eqref{e:distbound} and the fact that $\partial_{\tau} %v_{\zeta}$ is comparable to 
%$U_{\tau}$, we have that if $\tau=d(x_0+te_n,S)$, then 
Using Lemma \ref{l:n-1direction} we have 
\[
 V_{x_n}(x_0 + te_n,0) \geq \frac{1}{2}U_{\tau}(d(x_0+te_n,S),0) \geq \frac{1}{2}U_{\tau}(1,0)
 \geq c. 
\]
The second inequality comes from the fact that $U_{\tau}(\tau,0)$ is decreasing in $\tau$. Then 
\[
 (\overline{V} - u)(x_0,0) \geq V(x_0 + \mu/(16n) e_n,0) - V(x_0,0) 
 = \int_0^{\mu/(16n)} V_{x_n}(x_0 + t e_n,0) \ dt \geq c\mu/(16n)
\]
From the uniform $s$-H\"older continuity for both $u$ and $\overline{V}$ we obtain that 
\[
 \overline{V} - u_{\min} \geq \overline{V} - u \geq c \mu \quad \text{ in } B_{c \mu^{1/s}}.
\]
Then 
\[
\int_{B_1} (\overline{V}- u_{\min})|y| \geq c \mu^{C(s)}.
\]
for some large $C(s)$ depending on dimension and $s$. We then arrive at a contradiction from \eqref{e:sigabove} if $\sigma \leq \mu^{C(s)+1}$. Then no such point $(x_0,y_0)$ exists, and this concludes the proof.

\end{proof}
We now prove an analogous Lemma from above. 
\begin{lemma}  \label{l:supsol}
 Assume that $u$ satisfies $(H1)$-$(H3)$ and $u \leq V=V_{M, \xi', \zeta} \in \mathcal{V}_{\mu^{2/3}}$ in $B_1$. If 
 \[
  \frac{\zeta}{1-a} -\textit{tr}\ M \leq - \mu,
  \]
  and $\sigma \leq \mu^C$ with $C(s, n)$ large, then 
  \[
  u(x,y) \leq V_{M, \xi', \zeta}(X-t_0 e_n,y) \quad \text{ in } B_{1/4} \text{ for some } t_0 >0. 
  \]
\end{lemma}

\begin{proof}
 We proceed similarly as in Lemma \ref{l:subsol}. We assume by contradiction that there exists $(x_0,y_0) \in \overline{B}_{1/4}$ at which $V$ is tangent to $u$ from above. By the same arguments as in Lemma \ref{l:subsol}, we conclude $y_0=0$ and that $x_0 \in \overline{\{u(x,0)>0\}}$.
 %, and hence also $x_0 \in \overline{\{V(x,0)>0\}}$. 
 
 We define
 \[
  \underline{V}(x,y) := V_{\underline{M}, \xi, \underline{\zeta}} \left(x-\frac{\mu}{8n}e_n, y  \right),
  \quad \underline{\zeta} = \zeta+ \frac{\mu}{2n}, \quad \underline{M}= M - \frac{\mu}{2n} I. 
 \]
 As before we utilize \eqref{e:trap1} to conclude 
 \[
  \underline{V}(x,y) \geq V(x,y) \quad \text{ near } \partial B_1, 
  \quad \underline{V}(x,y)\leq V\left(x- \frac{\mu}{16n}e_n,y \right) \quad \text{ in } B_{1/2}. 
 \]
 We define
 \[
  u_{\max} := \max\{u, \underline{V}\}, \quad u_{\min}:= \min\{u, \underline{V}\},
 \]
 and we note that 
 \[
  u_{\max}=\underline{V}, \quad u_{\min}=u \quad \text{ near } \partial B_1. 
 \]
 Using $(H2)$ we have 
 \[
 J(u,B_1)\leq J(u_{\min},B_1) + \sigma. 
 \]
 Direct computations give  
 \[
 J(u_{\min},B_1) + J(u_{\max})= J(u,B_1) + J(\underline{V},B_1)\leq J(u_{\min},B_1)+ \sigma 
 + J(\underline{V},B_1),
 \]
 so that
 \[
 J(u_{\max},B_1)-J(\underline{V},B_1) \leq \sigma. 
 \]
 We now aim to obtain an inequality from below. We consider the modified functional 
 \[
 \tilde{J}(w,B_1):=J(w,B_1) - \mu^2 \int_{B_1} 2 w |y|,
 \]
 among all $H^1(a,B_1)$ functions $w$ satisfying $\max\{V,\underline{V}\} \geq w \geq \underline{V}$. From the analogue of Proposition \ref{p:viscmin} for supersolutions we have that $\underline{V}$ is the minimizer of the energy $\tilde{J}$. Thus, 
 \begin{equation} \label{e:supersolJ}
   J(u_{\max},B_1)- J(\underline{V},B_1) \geq \mu^2 \int_{B_1} (u_{\max}-\underline{V})|y|. 
 \end{equation}
  We now aim to show a bound of the form 
 \begin{equation} \label{e:lowermu}
     \int_{B_1} (u_{\max}-\underline{V})|y| \geq \mu^C
 \end{equation}
 which for small $\sigma$ will provide us with a contradiction. This is where the proof will differ from the subsolution case. We first consider the case in which $d((x_0,0), \{V=0\})\geq \mu^2$. 
 We recall from the beginning of the proof that $\underline{V}(x,y)\leq V\left(x-\frac{\mu}{16n}e_n, y \right)$ in $B_{1/2}$. Then
 \[
  (u-\underline{V} )(x_0,0) \geq u(x_0,0)-V(x_0-\mu/(16n)e_n,0) =V(x_0,0)- V(x_0-\mu/(16n)e_n,0).
 \]
 We now utilize the fact that $d((x_0,0),\{V=0\})\geq \mu^2$ and our bound for $V_{x_n}$ from Lemma \ref{l:n-1direction} to obtain 
 \[
  V(x_0,0)-V(x-\mu/(16n)e_n,0)=\int_{-\mu/(16n)}^0 V_{x_n}(x_0 - te_n, 0) \ dt 
  \geq \int_{-\mu^2}^0 V_{x_n}(x_0 - te_n, 0) \ dt 
  \geq c \mu^2. 
 \]
 Utilizing the $C^{0,s}$ bound on both $u$ and $\underline{V}$ as in Lemma \ref{l:subsol}, 
 we obtain \eqref{e:lowermu} and achieve a contradiction for small $\sigma$.

 We now consider the second case in which $d((x_0,0),\{V=0\}) < \mu^2$.
 From the construction of $V$, we have $V(x-\mu/(16n)e_n,y)=v_{\zeta}(\tau,\eta)\leq 2U(\tau,\eta)$ with $\tau \leq 32/n$, so that 
  \begin{equation}  \label{e:Vbelow} 
  V(x - \mu/(16n) e_n,y)\leq C\mu^{-s}|y|^{1-a} \quad \text{ in } B_{\mu^2}(x_0,0).     
 \end{equation}
 We denote $r=d((x_0,F(u))$. If $r\leq \mu^{3/2}$, then by $(H3)$ we have
 \[
 u \geq c\mu^{-s3/2}|y|^{1-a} \geq \frac{c}{2} \mu^{-s3/2}|y|^{1-a} + \underline{V}, \quad \text{ in }
 B_{\mu^{2}}(x_0,0),
 \]
 and the integral bound follows. We now consider the second subcase in which $r\geq \mu^{3/2}$.
 By the nondegeneracy property, we have 
 \[
 u(x,y) \geq c\mu^{s(3/2)} \quad \text{ for } (x,y) \in B_{\mu^{3/2}/2}(x_0,0).     
 \]
 Now from \eqref{e:Vbelow} we have if $(x,y) \in B_{\mu^{3/2}/2}(x_0,0)$ that
 \[
 \underline{V}(x,y) \leq C \mu^{-s} |y|^{2s} \leq C \mu^{-s} \mu^{3s} =C \mu^{2s} 
 \leq \frac{c}{2} \mu^{3/2}.   
 \]
 Then  
 \[
 u(x,y)-\underline{V}(x,y) \geq \frac{c}{2}\mu^{s(3/2)} \quad \text{ in } B_{\mu^{3/2}/2}(x_0,0).  
 \]
 The integral bound \eqref{e:lowermu} then follows, and we again achieve a contradiction for small $\mu$. 
 \end{proof}

\section{Compactness}\label{S:compactness}
In this section, we prove two properties that allow us to use a compactness argument and obtain a limiting solution. To do so, we will first
take a normalized hodograph transform of $u$, see \eqref{e:hodograph}. A Harnack inequality will give uniform convergence as $\epsilon \to 0$. Finally, the second property will show that the limit solution is a viscosity solution.

We begin with the following notation: we denote the half-hyperplane $P$ by 
\[
 P:= \{(x,0) \in \mathbb{R}^n \times \{0\}: x_n \leq 0\}
\]
and 
\[
L:= \{(x,0) \in \mathbb{R}^n \times \{0\}: x_n = 0\}.
\]
Also, we will consider translations of the solution $U$ denoted by 
\[
U_b := U(x+be_n,y),  \quad b \in \mathbb{R}. 
\]
We will assume that $u$ satisfies the assumptions $(H1)-(H3)$ as well as the $\epsilon$-flatness assumption 
\begin{equation} \label{e:epsilonflat}
 U(x-\epsilon e_n,y) \leq u(x,y) \leq U(x+\epsilon e_n, y) \quad \text{ in } B_1. 
\end{equation}

For the thin one-phase problem it is convenient to utilize the Hodograph transform. To motivate the use of the Hodograph transform, let us contrast the thin one-phase problem with the classical one-phase Alt-Caffarelli problem. The prototypical solution for the Alt-Caffarelli problem is $x_1^+$. If a solution $u$ satisfies $(x_1 - \epsilon)^+ \leq u \leq (x_1 + \epsilon)^+$, and one seeks to obtain an improvement in flatness, then PDE techniques make proving estimates in vertical translations fairly easy. In the positivity set, vertical and horizontal translations for $x_1$ are the same. However, in the thin one-phase problem, the prototypical minimizer $U$ has $C^{0,s}$ growth, so horizontal and vertical translations are not the same. Horizontal translations for $U$ are converted to vertical translations via the Hodograph transform, and PDE techniques can be applied to achieve these vertical translations in the Hodograph variables. A second and even more important reason for employing the Hodograph transform is that a convenient limiting equation (see Definition \ref{d:eq}) is obtained in the Hodograph variables.

We define the multivalued map $\tilde{u}$ as the $\epsilon$-normalized Hodograph transform of $u$ with respect to $U$ via the formula 
\begin{equation} \label{e:hodograph}
 U(x,y)=u(x - \epsilon \tilde{u}(x,y)e_n, y). 
\end{equation}
Since $U_{e_n}(x,y) > 0$ for $(x,y) \notin P$, then \eqref{e:epsilonflat} implies 
\[
|\tilde{u}| \leq 1 \quad \text{ in } B_{1-\epsilon} \setminus P. 
\]

In this section we will show that for small $\epsilon$, $\tilde{u}$ is approximated uniformly on compact subsets of $B_1$ by a viscosity solution to the linearized equation 
\begin{equation}  \label{e:linearized}
 \begin{cases}
    &\text{div}(|y|^a \nabla (U_n h))=0 \quad \text{ in } B_1 \setminus P, \\
    &|\nabla_r h|=0, \quad \text{ on } L. 
 \end{cases}
\end{equation}
Here $r=x_n^2 + y^2$, which is the distance to $L$. 
The definition of a viscosity solution $h$ to the linearized problem is given by the following. 
\begin{definition} \label{d:eq}
  $h$ is a solution to \eqref{e:linearized} if $h \in C(B_1)$, $h$ is even in $y$ and satisfies 
  \begin{itemize}
  \item div$(|y|^a \nabla (U_n h)) =0$ \quad in $B_1 \setminus P$. 
  \item Let $\phi$ be a continuous function $\phi$ which satisfies 
  \[
   \phi(x,y) = \phi(x_0',0,0) + a_1(x_0',0,0) \cdot (x' - x_0')+a_2(x_0',0,0)r + O(|x'-x_0'|^2 + r^{1+\gamma})
  \]
  for some $\gamma>0$ and $a_2(x_0',0,0)\neq 0$. If $a_2(x_0',0,0)>0$, then $\phi$ cannot touch $h$ by below at $(x_0',0,0)$, and if $a_2(x_0',0,0)<0$, then $\phi$ cannot touch $h$ by above at $(x_0',0,0)$. 
  \end{itemize}
\end{definition}

We will need some useful estimates for the $\epsilon$-normalized Hodograph transform. We state here the analogue of Proposition 2.8 in \cite{SS3}.  Notice that since both $U$ and $U_n$ are smooth away from the hyperplane $P$, the proof is identical to the case for $s=1/2$. 

\begin{proposition} \label{p:epsilonh}
  Let $\phi$ be a smooth function in $B_{\lambda}(x_0,y_0) \subset \mathbb{R}^{n+1} \setminus P$. For small $\epsilon$ we define $\phi_{\epsilon}$ by 
  \[
   U(x,y)=\phi_{\epsilon}(x-\epsilon \phi(x,y)e_n,y). 
  \]
  Then 
  \[
   \mathcal{L}_a \phi_{\epsilon} = \mathcal{L}_a \epsilon (U_n \phi) + O(\epsilon^2) \quad
   \text{ in } B_{\lambda/2}(x_0,y_0),
  \]
  with the remainder $O(\epsilon^2)$ depending on $\| \phi\|_{C^5}$ and $\lambda$. 
  In particular, if $\phi=Q/U_n$ with $Q$ a quadratic polynomial, then 
  \[
  \mathcal{L}_a \phi_{\epsilon}= \epsilon \mathcal{L}_a Q  + O(\epsilon^2) \quad 
  \text{ in } B_{\lambda/2}(x_0,y_0).
  \]
\end{proposition}

In order to obtain a limit solution $h$ from $\tilde{u}$, we need to prove two properties: firstly, a Harnack inequality that will provide uniform convergence as $\epsilon \to 0$. The second property will show that the limit solution $h$ is a viscosity solution. 

$(P1)$ \textit{Harnack Inequality:}

Given $\delta>0$, there exists $\epsilon_0=\epsilon_0(\delta)$ such that if $\epsilon \leq \epsilon_0$ and 
\[
 u \geq  U_b \ \text{ in } B_{r}(x_0,y_0) \subset B_1, \text{ and } |b|\leq \epsilon, r \geq \delta,  
\]
and
\[
u(x_1,y_1) \geq  U_{b+\tau \epsilon}(x_1, y_1) 
\quad \text{ for some } (x_1,y_1) \text{ with } B_{r/4}(x_1,y_1) \subset B_r(x_0,y_0) \setminus P, \text{ and } \tau \in [\delta, 1], 
\]
then 
\[
u \geq  U_{b+c\tau \epsilon} \quad \text{ in } B_{r/2}(x_0,y_0) 
\]
 for some constant $c>0$. 

 An analogous property holds when $\leq$ and $-\tau$ replace $\geq$ and $\tau$ respectively. 

$(P2)$ \textit{ Viscosity Property:}

Given $\delta>0$, there exists $\epsilon_0 = \epsilon_0(\delta)$ such that if $\epsilon\leq \epsilon_0$, then 

$(a)$ we cannot have $u(x_0,y_0)=q(x_0,y_0)$ and $u \geq q$ in $B_{\delta}(x_0,y_0) \subset B_1\setminus P$ where
\[
q \in C^2(B_{\delta}(x_0,y_0)) \quad \|D^2 q \| \leq \delta^{-1}, 
\quad \mathcal{L}_a q \geq \delta \epsilon. 
\]

$(b)$ we cannot have $u \geq V$ in $B_{\delta}(x_0,y_0) \subset B_1$ with $(x_0,y_0) \in L$ and $V$ tangent by below to $u$ in $B_{\delta/4}(x_0,y_0)$ where $V$ is a translation of a function $V_{M, \xi', \zeta}(x+te_n, y) \in \mathcal{V}_{\delta^{-1} \epsilon}$, 
\[
V(x,y):=V_{M, \xi', \zeta}(x+te_n, y), \quad \text{ with } \frac{\zeta}{1-a} - \text{tr} M \geq \epsilon.  
\]

Similarly, $a$ and $b$ hold when we compare $u$ with functions $q,V$ by above and $\mathcal{L}_a q \leq -\delta \epsilon$ and $\zeta/(1-a) - \text{tr} M \leq - \epsilon$.

We will now show that the properties $(P1)$ and $(P2)$ are satisfied. 
\begin{theorem} \label{t:p1p2}
Properties $(P1)$ and $(P2)$ are satisfied. 
\end{theorem}

\begin{proof}
 We fix $\delta >0$ and split the proof into several cases. 

 \textit{Case 1}: $B_r(x_0,y_0)\cap \{y=0\} = \emptyset$. 

 We have that $d((x_0,y_0),L)>\delta$. By rescaling at the point $(x_0,0)$ we may assume that $B_{7r/8}(x_0,y_0) \cap \{|y| \leq c\}=\emptyset$. Since both $u$ and $U$ are $a$-harmonic, and since we are at uniform distance from the thin space, then vertical and horizontal translations of $U$ are comparable. Then property $(P1)$ is a simple consequence of the Harnack inequality for the operator $\mathcal{L}_a$. Property $(P2a)$ is a simple consequence of the comparison principle for $a$-harmonic functions, and $(P2b)$ is not applicable in this case.  

 \textit{Case 2}: $y_0 =0$ and $\mathcal{B}_r(x_0) \subset P$.

 Once again by rescaling, we may assume that $d(B_{7r/8}(x_0,0) , L)>c$. We have that
 $(u- U_b)(x_1,y_1)\geq \epsilon U_b(x_1,y_1).$ By the boundary Harnack principle for $a$-harmonic functions (Proposition \ref{p:bhpa}), we have 
 \[
  \epsilon \leq \frac{(u-U_b)(x_1,y_1)}{U_b(x_1,y_1)} \leq C  \frac{(u-U_b)(x,y)}{U_b(x,y)} \quad \text{ for any } (x,y) \in B_{r/2}(x_0,0).   
 \]
 Since we are a uniform distance from $L$, the vertical and horizontal shifts of $U_b$ are comparable. 
 Property $(P1)$ immediately follows. Property $(P2a)$ follows by the comparison principle for $a$-harmonic functions. 

 \textit{Case 3}: $y_0=0$ and $B_r(x_0,0) \cap P=\emptyset$. 

 As in cases 1 and 2, we may rescale and assume $d(B_{7r/8}(x_0,0) , L)>c$. If $u$ were $a$-harmonic, then the proof would be identical to case 1. Instead, we need to use the almost minimality of $u$ to compare it with the $a$-harmonic replacement $v$ which solves 
 \[
  \begin{cases}
   &\mathcal{L}_a v =0  \quad \text{ in } B_{7r/8}(x_0,0), \\
   &v=u \quad \text{ on } \partial B_{7r/8}(x_0,0). 
  \end{cases}
 \]
 By the almost minimality of $u$ (property $H2$), and also the $a$-harmonicity of $v$, we have 
 \[
  \int_{B_{7r/8}(x_0,0)} |\nabla (u-v)|^2 |y|^a \leq \sigma. 
 \]
From the weighted ($A_2$) Poincare inequality for $w=u-v$ we have 
\[
 \| w  \|_{L^2(a,B_{3r/4}(x_0,0))} \leq C r \| \nabla w \|_{L^2(a,B_{7r/8}(x_0,0)}.
\]
Assume now that $w(x_1,y_1)\geq 2\mu$ with $(x_1,y_1) \in B_{3r/4}$. From the $s$-H\"older continuity of both $u$ and $v$ we have for small $\mu$ that 
\[
 w \geq \mu \text{ on } B_{\mu^{2/s}}(x_1,y_1). 
\]
 Thus, 
 \[
  Cr^2 \sigma \geq \int_{B_{7r/8}(x_0,0)} w^2 |y|^a \geq \mu^{(n+a)2/s} \mu^2 =\mu^{(2n+1+a)/s}.
 \]
 Then 
 \[
 |u-v| \leq 2(C r^2 \sigma)^{\frac{s}{2(n+1+a)}} \leq \epsilon^2 \quad \text{ in } B_{3r/4}(x_0,0).
\]
% Using Chebyshev's inequality we have
% \[
%  \int \leq \sqrt{Cr^2 \sigma}.
% \]
% If $(x_2,y_2)$ is a distance $R$ from $\{w^2 > \sqrt{Cr^2 \sigma}\}$, then %necessarily 
% \[
% C R^{n+1+a} \leq \sqrt{Cr^2 \sigma}.  
% \]
% Using the $s$-H\"older continuity of both $u$ and $v$, we obtain 
% \[
%  |w(x_2,y_2)| \leq (C r^2 \sigma)^{\frac{s}{2(n+1+a)}} + \sqrt{Cr^2 \sigma} %\leq (C r^2 \sigma)^{\frac{s}{2(n+1+a)}}. 
% \]
% Then 
%\[
% |u-v| \leq (C r^2 \sigma)^{\frac{s}{2(n+1+a)}} \leq \epsilon^2 \quad \text{ %in } B_{3r/4}(x_0,0).
%\]
Recall that we require $\sigma \leq \epsilon^C$ for $C$ large. 
 Since $u$ is approximated by an $a$-harmonic function up to order $\epsilon^2$, it follows from the arguments in Case 1 applied to the $a$-harmonic replacement $v$. This proves Property $(P1)$. 

 We now need to show property $(P2)$ part $a$. To do so we utilize the fact that
 \[
 \mathcal{L}_a  \delta \epsilon \frac{1}{2n(1+a)} (\delta^2 - |x-x_0|^2 - y^2) = -\delta \epsilon, 
 \]
 with zero boundary data on $\partial B_{\delta}(x_0,0)$. 
 Since $\mathcal{L}_a (v - q) \leq -\delta \epsilon$ and $v-q \geq 0$ on $\partial B_{\delta}(x_0,0)$, it follows that
 \[
  (v-q)(x_0,0) \geq \frac{\epsilon \delta^2}{2n(1+a)}. 
 \]
 Since $v$ approximates $u$ to order $\epsilon^2$, we have that $(u-q)(x_0,0)>0$. 
 
\textit{Case 4}: $(x_0,y_0) \in L$. 

We may rescale and translate to assume that $r=1$ and $(x_0,y_0)=(0,0)$. We will first prove property $(P1)$. To avoid a horizontal translation in all the notation (with corresponding vertical translations in the hodograph transform), we will assume for simplicity that $b=0$. As a consequence we have $u \geq U$ in $B_1$ and also $\tilde{u} \geq 0$ in $B_{3/4}$. From cases $(1)$-$(3)$ and utilizing a Harnack chain, we have that 
\[
u \geq U_{c_0 \tau \epsilon} \quad \text{ outside of } \ \{|x'| \leq 3/4, \ |(x_n,y)| \leq k_0\}. 
\]
With the constant $c_0$ depending on $k_0$. In terms of the hodograph transform this gives 
\begin{equation}  \label{e:tildebound}
\tilde{u} \geq c_0\tau  \quad \text{ outside of }\  \{|x'| \leq 3/4, \ |(x_n,y)| \leq k_0\}.
\end{equation}
We now consider the quadratic polynomial 
\[
Q(x,y):= -\frac{1}{2} |x'|^2 + \frac{n}{2s}(x_n^2 + y^2).
\]
We have the associated subsolution 
\[
 V:=V_{\epsilon I,0,2n\epsilon}.
\]
Utilizing Proposition \ref{p:hodograph} and recalling that $\tilde{V}$ is not $\epsilon$-normalized, we have 
\[
 \tilde{V}=Q +O(\epsilon). 
\]
From \eqref{e:tildebound} we have 
\[
\tilde{u} \geq \frac{c_0 s}{n} \tau  \quad \text{ in }\ (B_{3/4}\setminus B_{1/16}) \cap  \{|x'| \leq 3/4, \ |(x_n,y)| > k_0\}.
\]
Choosing $k_0$ small and $N$ large (both depending on $s$), we have 
\[
 2^{-N} + Q < -2^{-(N+1)} \quad \text{ in }  \ (B_{3/4}\setminus B_{1/16}) \cap  \{|x'| \leq 3/4, \ |(x_n,y)| > k_0\}.
\]
Then we obtain 
\[
\tilde{u} \geq \frac{c_0 s}{n} \tau(2^{-N}+Q) \quad \text{ in } B_{3/4} \setminus B_{1/16}. 
\]
Then $u \geq V(x+c_1 \tau \epsilon e_n,y)$ in the annulus $B_{3/4} \setminus B_{1/16}$
with $c_1 = c_0 s/(n2^{N+2})$. 
We now apply Lemma \ref{l:subsol} to $V(x+c_1 \tau \epsilon e_n,y)$ to conclude that 
\[
u \geq V(x+c_1 \tau \epsilon e_n, y) \quad \text{ in } B_{3/4}. 
\]
This proves property $(P1)$. 

To prove property $(P2b)$ we assume by contradiction that $(x_0,y_0)=0$ and that $u\geq V$ in $B_{\delta}$ with $V$ satisfying the hypotheses of $(P2)$ part $(b)$. We rescale with $u_{\delta}=\delta^{-s} u(\delta x)$  and $\overline{V}= \delta^{-s} V(\delta x)$. Then $\overline{V}$ has coefficients $\overline{a}=\delta a, \overline{M}=\delta M, \overline{\xi}=\xi, \overline{\sigma}=\delta^{n-1}\sigma$. Then $\overline{V} \in \mathcal{V}_{\delta^{-1}\epsilon}$ and is tangent to  $\overline{u}$ by below at the origin. This contradicts Lemma \ref{l:subsol}.  

\textit{Case 5: $y_0 \neq 0$ and $(B_r(x_0,y_0)\cap\{y=0\})\neq \emptyset$}

Let $r_1$ denote the radius of the $n$-dimensional ball realized as the intersection of $B_r(x_0,y_0)$ with the thin space. That is, 
\[
 \{|(x_0,0) - (x,0)| < r_1\} = B_r(x_0,y_0)\cap\{y=0\}. 
\]
If $B_{r_1}(x_0,0) \cap B_{r/2}(x_0,y_0) = \emptyset$, then $B_{r/2}(x_0,y_0)$ is far enough from the thin space to be treated exactly as in Case 1. 

Now assume $B_{r_1}(x_0,0) \cap B_{r/2}(x_0,y_0) \neq \emptyset$. If $(x_1,y_1) \in B_{3r_1/4}(x_0,0)$, then we can utilize Cases (2-4), and then a Harnack chain combined with ideas from Case 1 to conclude the result. If $(x_1,y_1) \notin B_{3r_1/4}(x_0,0)$, then we can utilize the ideas from Case 1, and then a Harnack chain and cases (2-4) to conclude the result.  
\end{proof}

\section{Free Boundary Regularity}\label{S:freebdryreg}

In this section we prove our final main result, a flatness implies $C^{1,\gamma}$ regularity theorem, see Theorem \ref{t:reg}. To begin with, we prove a compactness result that tells us that as $\epsilon_k\rightarrow 0$, the functions obtained thorough the normalized hodograph transform converge to a solution of \eqref{e:linearized}.

\begin{lemma}  \label{l:trapped} 
 Assume that $u_k$ satisfies $(H1)$-$(H3)$ and that 
 \[
  U(x-\epsilon_k e_n, y) \leq u_k \leq U(x+\epsilon_k e_n, y) \quad \text{ in } B_1. 
 \]
 If $\epsilon_k \to 0$, and $\tilde{u}_k$ is the multi-valued function defined by 
 \[
   U(x,y)=u_k(x- \epsilon_k \tilde{u}_k e_n, y),   
 \]
 then $\tilde{u}_k$ has a convergent subsequence which converges uniformly to a limiting function $h$ solving \eqref{e:linearized}. 
\end{lemma}

\begin{proof}
 We let $\epsilon_k \leq \epsilon_0(2^{-k})$ with $\epsilon_0(\delta)$ as given in properties $(P1)$ and $(P2)$. Then necessarily $\sigma_k \to 0$ as well. Property $(P1)$ ensures that $\tilde{u}_k$ satisfies the Harnack inequality in balls of size $r$ included in $B_1$ with $2^{-k}\leq r \leq 1$. Then $\tilde{u}_k$ all have the same uniform H\"older bound, and we can therefore extract a convergent subsequence going to a limiting function $h$ which is H\"older continuous in $B_1$. 

 We now show that $h$ is a solution to \eqref{e:linearized}. We first show that $U_n h$ is $a$-harmonic for any $(x_0,y_0) \in B_1 \setminus P$. Since $U_n h$ is even, it suffices to check in a viscosity sense by touching above and below with quadratic polynomials. Suppose that $Q$ is a quadratic polynomial that touches $U_n h$ strictly by below at a point $(x_0,y_0) \in B_1 \setminus P$, and that $\mathcal{L}_a Q >0$. Since $\tilde{u}_k$ converges uniformly, we have that a translation of $Q/U_n$ touches $\tilde{u}_k$ by below near $(x_0,y_0)$. From $(3.7)$ in \cite{SSS} we have that $q_k$ defined by 
 \[
 U(x,y)=q_k\left(x- \epsilon_k \frac{Q}{U_n}e_n,y \right)
 \]
 touches $u$ by below at a point $(x_k,y_k)$ and is below $u_k$ in a fixed neighborhood of $(x_0,y_0)$. By Proposition \ref{p:epsilonh} we have that 
 \[
  q_k \in C^2(B_{\delta}(x_k,y_k)), \quad \|D^2 q_k \| \leq \delta^{-1}, \quad 
  \mathcal{L}_a q_k = \epsilon_k \mathcal{L}_a Q + O(\epsilon_k^2) \geq \delta \epsilon_k,
 \]
 for a fixed $\delta$ depending only on $(x_0,y_0)$ and $Q$. By property $(P2)$ part $a$ we obtain a contradiction. 

 Finally, we now show that $h$ satisfies the zero-Neumann boundary condition. Suppose by contradiction that there exists $\phi$ which touches $h$ by below with $\phi(0)=0$ and such that 
 \[
  \phi(x,y)=\xi' \cdot x' + a_2 r+ O(|x'|^2 + r^{1+\gamma})
 \]
 for some $\gamma>0$ and with $a_2>0$. Then we can find a constant $\alpha \geq 1$ large and depending on $\phi$ such that the quadratic polynomial 
 \[
  Q(x,y)=\xi' \cdot x' -\frac{\alpha}{2} |x'|^2 + n \alpha r^2
 \]
 touches $h$ strictly by below at $0$ in $B_{2\delta}$ for some sufficiently small $\delta$. We let
 \[
 V:=V_{-\epsilon \alpha I, \epsilon \xi', -4sn\epsilon \alpha} \in V_{\delta^{-1}\epsilon},
 \]
 which satisfies the conditions of property $(P2)$ part $(b)$. By Proposition \ref{p:hodograph} the $\epsilon$-rescaling of the hodograph transform of $V$ satisfies 
 \[
 \tilde{V}=Q + O(\epsilon).
 \]
 Since $\tilde{u}_k$ converges uniformly to $h$, we obtain that a translation of $\tilde{V}$ touches $\tilde{u}_k$ by below. Consequently, $V$ touches $u_k$ by below at some point $(x_k,y_k) \to 0$ and $u_k$ is above this translation in $B_{\delta}$. This contradicts property $(P2)$ part $b$. 
\end{proof}

This next proposition gives an estimate for the solution $h$ of the limiting equation. 
\begin{proposition} \label{p:limitreg}
    Assume $h$ solves \eqref{e:linearized} in $B_1$ with $h(0)=0$ and $|h|\leq 1$. There exists a universal constant $\gamma_0>0$, such that if $\gamma \in (0,\gamma_0)$, then there exists a constant $\rho=\rho(\gamma)$ such that   
    \begin{equation} \label{e:limitreg}
     |h(x,y)-(h(0,0) + \xi' \cdot x')|\leq \frac{1}{4}\rho^{1+\gamma} \quad \text{ in } B_{2\rho}. 
    \end{equation}
    for some $\xi' \in \mathbb{R}^{n-1} \times \{0\} \times \{0\}$. 
\end{proposition}

\begin{proof}
    The result is a consequence of Theorem 6.1 in \cite{SSS}, where the estimate 
    \[
     |h(x,y)-(h(0,0) + \xi' \cdot x')|\leq C |(x,y)|^{1+\gamma} \quad \text{ in } B_{1/2},
    \]
    is proved for universal constants $C$ and $\gamma$. If $\gamma_1 < \gamma$, then we can choose $\rho$
    small enough such that 
    \[
     C(2\rho)^{1+\gamma} \leq \frac{1}{4}\rho^{1+\gamma_1},
    \]
    and the conclusion follows. 
\end{proof}

This next Lemma will transfer properties of $h$ to $u$. 
\begin{lemma} \label{l:approximate}
 Let $u$ satisfy $(H1)-(H3)$. Assume that $0 \in F(u)$ and 
 \begin{equation} \label{e:trap3}
  U(x-\epsilon e_n, y) \leq u \leq U(x+ \epsilon e_n, y). 
 \end{equation}
 Then there exists a $\gamma \in  (0,1)$ depending on $n,\beta, s$ such that if $\sigma \leq \epsilon^C$ for some large constant $C=C(\beta,n,s)$, then 
 \begin{equation} \label{e:trap2}
 U(x\cdot \nu - \epsilon \rho^{1+\gamma}) \leq u \leq  U(x\cdot \nu + \epsilon \rho^{1+\gamma}) \quad \text{ in } B_{\rho},
 \end{equation}
 for some direction $\nu \in \partial B'_1$, provided that $\epsilon \leq \epsilon_0(\gamma,n,s,\beta)$. 
\end{lemma}

\begin{proof}
  Suppose by contradiction that the Lemma is not true. Then there exists a sequence $u_k$ satisfying $(H1)-(H3)$ as well as \eqref{e:trap3} with $\epsilon_k \to 0$ and $\sigma \leq \epsilon_k^C$. Also, $0 \in F(u_k)$; however, by contradiction we assume that no such $\nu$ exists so that \eqref{e:trap2} holds. Now from Lemma \ref{l:trapped} we have that 
  $\tilde{u}_k$ converges uniformly to $h$ a solution of \eqref{e:hodograph}. From Proposition \ref{p:limitreg} we have that $h$ satisfies \eqref{e:limitreg}. Since $0 \in F(u_k)$ for all $k$ we have that $\tilde{u}_k(0)=0$, so that $h(0)=0$ . Since $\tilde{u}_k$ converges uniformly to $h$ we have that for $k$ large enough
  \[
  |\tilde{u}_k(x,y)-  \xi' \cdot x'|\leq \frac{1}{4}\rho^{1+\gamma} \quad \text{ in } B_{2\rho}. 
  \]
  From the same arguments in Lemma 7.2 in \cite{DR12} we can relate the vertical bounds in the hodograph to a horizontal shift in the original variables: 
  \[
   U(x'\cdot \xi' - \epsilon_k \frac{1}{2}\rho^{1+\gamma},y) 
   \leq 
   u(x,y) 
   \leq 
   U(x'\cdot \xi' + \epsilon_k \frac{1}{2}\rho^{1+\gamma},y) 
   \quad \text{ in } B_{\rho}. 
  \]
  We then arrive at a contradiction that no such $\nu$ exists since $\xi=\nu$ will suffice.  
\end{proof}

We now conclude the regularity of the free boundary at flat points. 
\begin{theorem} \label{t:reg}
 Let $u$ be an almost minimizer to $J$ in $B_1$ with constant $\kappa$ and exponent $\beta$. Assume that $\|u\|_{C^{0,1/2}(B_1)} \leq C$. Assume that $|u-U|\leq \tau_0$ in $B_1$. If $\tau_0$ and $\kappa$ are small enough depending on $\beta,n,s$, then $F(u)$ is $C^{1,\gamma_0}$ in $B_{1/2}$ for some $\gamma_0>0$ and depending on $n,\beta,s$. 
\end{theorem}

\begin{proof}
 We assume that $0 \in F(u)$. Otherwise, we translate and rescale to the ball of radius $1$. 
 If $\tau_0$ is chosen small enough, then 
 \[
  U(x-\epsilon e_n, y) \leq u \leq U(x+\epsilon e_n, y). 
 \]
 We will now proceed with induction to prove there is a sequence of radii $r = \eta^k$ with unit directions in $\nu_r \in B'_1$ such that 
 \[
  U(x \cdot \nu_r - \epsilon_0 r^{1+\gamma},y)\leq u(x,y) \leq 
  U(x\cdot \nu_r + \epsilon_0 r^{1+\gamma},y) \quad \text{ in }  B_r. 
 \]
 The base case was shown directly above in Lemma \ref{l:approximate}. The rescaling
 \[
 u_r(x,y):=\frac{u(rx,ry)}{r^s}
 \]
 satisfies the $\sigma$ estimate $(H2)$ with a new $\sigma=(\kappa r^{\beta})^{1/3}$. Also, by the induction step, $u_r$ satisfies the flatness condition for Lemma \ref{l:approximate} with $\epsilon=\epsilon r^{\gamma}$. To obtain $\sigma \leq \epsilon^C$, we need 
 \[
 \kappa r^{\beta} \leq (\epsilon_0 r^{\gamma})^3 
 \]
 which is possible by choosing both $\kappa$ and $\gamma$ small enough. 
\end{proof}

\section{Future directions}

There are many potential ways to continue the study of these almost minimizers. A natural question is whether a Weiss type monotonicity formula holds. While the standard techniques presented obstructions, it is an avenue one might consider exploring. The Weiss type monotonicity formula could be useful in a critical analysis of the singular set. This was carried out in \cite{EE} for minimizers of the Alt-Caffarelli problem (when $s=1$). 

\appendix 
\section{Minimizers are viscosity solutions}
We employ standard techniques to show that minimizers are viscosity solutions. This will show a necessary step in the proof of Lemma \ref{l:subsol}. As this result is not explicitly shown in the literature, we provide the result here. 

\begin{proposition} \label{p:viscmin}
    Let $V$ and $\overline{V}$ be given as in the proof of Lemma \ref{l:subsol}. Then $\overline{V}$ is the minimizer of the functional
    \[
    \tilde{J}(w,B_1)=J(w,B_1)+ \int_{B_1}  2\mu^2 w |y|,
    \]
    among all $w \in H^1(a,B_1)$ satisfying $\min\{V, \overline{V}\}\leq w \leq \overline{V}$.
\end{proposition}

\begin{proof}
 Let $\tilde{w}$ be the minimizer, and let $V_t=V_{\overline{M}, \xi', \overline{\zeta}}(x+te_n,y)$ with $t \in [0,\mu/(8n)]$. Notice that $V_0 \leq \min \{V, \overline{V}\}$ and $V_{\mu/(8n)}=\overline{V}$. Furthermore, $V_t <\overline{V}$ on $\partial B_1 \cap \{\overline{V}>0\}$. We will show that if $V_t$ touches $\tilde{w}$ by below on $\{\tilde{w}>0\}$ or $F(\tilde{w})$, then necessarily $\tilde{w}=V_t$, so that $\overline{V}=\tilde{w}$. 
 
 Suppose $V_t(X)=\tilde{w}(X)>0$. Now $|y|^a \mathcal{L}_a \tilde{w}=\mu^2 |y|$. Since $\mathcal{L}_a V_t \geq \mathcal{L}_a \tilde{w}$ in $\{\tilde{w}>0\}$, then from the strong maximum principle, we conclude that $V_t \equiv \tilde{w}$. 
 
 Now suppose that $x \in F(V_t) \cap F(\tilde{w})$. From the standard proofs of \cite{A} adapted to this nonhomogeneous situation, we have that $\tilde{w}$ has both the $C^{0,s}$ regularity as well as the nondegeneracy. Since $V_t$ touches $\tilde{w}$ by below at $x$, then every blow-up of $\tilde{w}$ at $x$ must be the same rotation of the prototypical $2D$ solution. Then 
 \[
  \frac{\partial \tilde{w} - V_t}{\partial \tau^s}(x,0)=0. 
 \]
 Consequently, 
 \[
 \lim_{y \to 0}y^a \partial_y (\tilde{w} - V_t)(x,y)=0.  
 \]
 Since $\tilde{w}-V_t \geq 0$, $\tilde{w}(x,0)-V_t(x,0)=0$, then this will violate the Hopf-type lemma \eqref{e:hopf}. 
\end{proof}

\section*{Declarations}
\subsection*{Funding}
M.A. has been partially supported by the Simons Grant 637757. M.S.V.G has been partially supported by the NSF grant DMS-2054282.

\subsection*{Potential conflicts of interest}
The authors have no conflicts of interest to declare that are relevant to the content of this article.

\subsection*{Research involving human participants and/or animals}
Not applicable.

\subsection*{Informed consent}
Not applicable.

\subsection*{Data availability}
Not applicable.

\bibliographystyle{plain}
\bibliography{Bibliography.bib}

\end{document}